\documentclass[12pt]{amsart} \usepackage{hyperref}
\usepackage{amsmath,amssymb,amscd,amsthm,amsfonts}
\usepackage[all,cmtip]{xy}
\usepackage[numbers]{natbib}

\newtheorem{thm}{Theorem}[section]
\newtheorem{lemma}[thm]{Lemma}
\newtheorem{prop}[thm]{Proposition}
\newtheorem{cor}[thm]{Corollary}
\theoremstyle{definition}
\newtheorem{defn}[thm]{Definition}
\theoremstyle{remark}

\newtheorem{rmk}[thm]{Remark}
\numberwithin{equation}{section}

\newcommand{\BA}{{\mathbb {A}}}

\newcommand{\CA}{{\mathcal {A}}}

\newcommand{\cusp}{{\mathrm{cusp}}}

\newcommand{\GL}{{\mathrm{GL}}}

\newcommand{\Hom}{{\mathrm{Hom}}}

\newcommand{\Ind}{{\mathrm{Ind}}}

\newcommand{\Ker}{{\mathrm{Ker}}}

\newcommand{\Lie}{{\mathrm{Lie}}}

\renewcommand{\Re}{{\mathrm{Re}}}

\newcommand{\Res}{{\mathrm{Res}}}
\newcommand{\res}{{\mathrm{res}}}

\newcommand{\Sp}{{\mathrm{Sp}}}
\newcommand{\Stab}{{\mathrm{Stab}}}

\newcommand{\Span}{{\mathrm{Span}}}
\newcommand{\tr}{{\mathrm{tr}}}

\newcommand{\ur}{{\mathrm{ur}}}

\newcommand{\vol}{{\mathrm{vol}}}

\newcommand{\bs}{\backslash}

\newcommand{\A}{\mathbb{A}}
\newcommand{\Z}{\mathbb{Z}} \newcommand{\Q}{\mathbb{Q}}
\newcommand{\R}{\mathbb{R}} \newcommand{\C}{\mathbb{C}}
\newcommand{\F}{\mathbb{F}} \newcommand{\cO}{\mathcal{O}}
\newcommand{\cS}{\mathcal{S}}\newcommand{\cH}{\mathcal{H}}
\newcommand{\cA}{\mathcal{A}}

\newcommand{\cE}{\mathcal {E}}

\newcommand{\cP}{\mathcal {P}}

\newcommand{\fa}{\mathfrak {a}}
\DeclareMathOperator{\FO}{\mathrm {FO}}
\DeclareMathOperator{\LO}{\mathrm {LO}}

\newcommand{\half}{\frac {1} {2}}

\newcommand{\trpz}[1]{{{}^{\mathrm{t}}\negthinspace{#1}}}

\DeclareMathOperator{\Alt}{\mathrm{Alt}}

\newcommand{\isom}{\cong}

\DeclareMathOperator{\diag}{\mathrm{diag}}
\DeclareMathOperator{\lmod}{\backslash}
\let\Re\undefined
\DeclareMathOperator{\Re}{\mathrm{Re}}

\newcommand{\smatrix}[4]{\ensuremath\bigl( \begin{smallmatrix}
#1&#2\\ #3&#4
\end{smallmatrix} \bigr)}
\newcommand{\form}[2]{\langle{#1},{#2}\rangle}
 \usepackage[left=1.5in,
right=1.5in, top=1.5in, bottom=1.5in]{geometry}
\newcommand{\rf}{\mathrm {rf}}
\newcommand{\ram}{\mathrm {ram}}
\DeclareMathOperator{\Rat}{\mathrm{Rat}}
\makeatletter

\newcommand{\Rmnum}[1]{\expandafter\@slowromancap\romannumeral #1@}
\makeatother
\begin{document} \title[Periods and $(\chi,b)$-factors of Cuspidal
Forms of $\Sp(2n)$] {Periods and $(\chi,b)$-factors of Cuspidal
Automorphic Forms of Symplectic Groups}

\author{Dihua Jiang} \address{School of Mathematics, University of
Minnesota, Minneapolis, MN 55455, USA} \email{dhjiang@math.umn.edu}

\author{Chenyan Wu} \address{Shanghai Center for Mathematical Sciences, Fudan University, Shanghai, China, 200433}
\email{chywu@fudan.edu.cn}

\thanks{Corresponding author. Telephone number: .}  

\thanks{The research of the first named author is supported in part by
  the NSF Grants DMS--1301567 and DMS--1600685. The research of the second named author is supported in part by National Natural Science Foundation of China (\#11601087) and by Program of Shanghai Academic/Technology Research Leader (\#16XD1400400).}

\date{\today}

\begin{abstract} In this paper, we introduce a new family of period
integrals attached to irreducible cuspidal automorphic representations
$\sigma$ of symplectic groups $\Sp_{2n}(\BA)$, which detects the right-most pole of the $L$-function
$L(s,\sigma\times\chi)$ for some  character $\chi$ of
$F^\times\bs\BA^\times$ of order at most $2$, and hence the occurrence of a
simple global Arthur parameter $(\chi,b)$ in the global Arthur
parameter $\psi$ attached to $\sigma$. We also give a characterisation of first occurrences of theta correspondence by (regularised) period integrals of residues of certain Eisenstein series.
\end{abstract}
\maketitle{}

\section{Introduction}

Let $\sigma$ be an irreducible cuspidal
automorphic representation of a symplectic group $\Sp_{2n}(\BA)$ with
$\BA$ the ring of adeles of a number field $F$. The tensor product
$L$-functions $L(s,\sigma\times\tau)$, with $\tau$ varying in the set of
equivalence classes of irreducible unitary cuspidal automorphic
representation of a general linear group $\GL_a(\BA)$, are important
invariants associated to $\sigma$. R. Langlands proves that this family of
$L$-functions converges absolutely when the real part of $s$ is large and
has meromorphic continuation to the whole complex plane
(\cite{MR0419366}). It remains to show that they satisfy the standard
functional equation and have finitely many poles on the real line with
$s\geq\frac{1}{2}$. Various approaches have been undertaken in order
to establish these expected analytic properties of those
$L$-functions.  The recent progress includes the work of J. Arthur
(\cite{MR3135650}) for the case when $\sigma$ has a global generic Arthur
parameter,  the Langlands-Shahidi method in
\cite{MR2683009}, and the work (\cite{MR2906909} and \cite{MR3187349})
for general $\sigma$.  As discussed in (\cite[Section
8]{jiang14:_autom_integ_trans_class_group_i}), it is not hard to check via the relevant theory on automorphic representations
that the location of the poles of the $L$-functions
$L(s,\sigma\times\tau)$ 
essentially detects the occurrence of the simple global Arthur
parameter $(\tau,b)$ in the global Arthur parameter $\psi$ of $\sigma$.

In this paper, we consider a special subfamily of those $L$-functions,
that is, the $L$-functions $L(s,\sigma\times\chi)$ with $\chi$ being
 characters of $F^\times\bs\BA^\times$ of order at most $2$. There are
three different types of global zeta integrals which represents this
family of $L$-functions. First we have the doubling method which is given in
\cite{MR892097} by I. Piatetski-Shapiro and S. Rallis. It  is regarded as
the natural generalisation of the method of J. Tate and of
Godement-Jacquet.  In their {\sl new way to get Euler products} in
\cite{MR965059}, Piatetski-Shapiro and Rallis give a second way, using an idea based on the
property that cuspidal automorphic representations are
non-singular. The theory of non-singular automorphic forms has been
established through the work of R. Howe for symplectic groups
(\cite{MR644835}) and of J.-S. Li for general classical groups
(\cite{MR1168122}). The third method outlined in \cite{MR2906909}
uses the Fourier-Jacobi coefficients of cuspidal automorphic
forms. The idea is based on the extension of the Bernstein-Zelevinsky
derivatives from general linear group to classical groups, which has
now been used to produce the automorphic descent method of Ginzburg-Rallis-Soudry in \cite{MR2848523}.

When this family of $L$-functions is represented by the doubling
zeta integrals, the poles of the $L$-functions are closely related to
the theta correspondences from symplectic groups to the towers of
orthogonal groups via the regularised Siegel-Weil formula of Kudla-Rallis and the Rallis inner product formula (\cite{MR1289491}).
A complete theory using the values or poles of the $L$-functions
$L(s,\sigma\times\chi)$ to detect the local-global relation of the theta
correspondences, initiated by Rallis in his series of work, has been given in recent papers of S. Yamana in \cite{MR3211043}
and of Gan-Qiu-Takeda in
\cite{gan14:_regul_siegel_weil_formul_secon}. We refer to
\cite{gan14:_regul_siegel_weil_formul_secon} for the complete story of
the theory and the relevant references on the topic.

One of  main results of this paper is a characterisation of the lowest occurrences of theta correspondence in terms of poles of certain Eisenstein series $E(g,f_{\chi\otimes\sigma},s)$ constructed from $(\sigma,\chi)$ (Thm.~\ref{thm:Eis-pole-LO-strengthened}). A more intricate result concerns non-vanishing of period integrals of residue of the $E(g,f_{\chi\otimes\sigma},s)$ and the first occurrences of theta correspondence (Cor.~\ref{thm:FO-Eis-pole}). We note that the first occurrence is a finer invariant of $(\sigma,\chi)$ than the lowest occurrence. (See Sec.~\ref{sec:fo} for the difference.) These period integrals \eqref{eq:period-residue-2} are integrations of  residue of $E(g,f_{\chi\otimes\sigma},s)$ twisted by a theta function over automorphic quotients of symplectic subgroups or Jacobi subgroups of $\Sp_{2(n+1)}(\A)$. Arthur truncation  is applied to regularise the integrals, but it turns out that the values do not depend on the truncation parameter. As a result, we have a notion of distinction by symplectic subgroups or Jacobi subgroups. We say that those automorphic forms are distinguished by a symplectic subgroup or a Jacobi subgroup if the relevant (regularised) period integrals do not vanish. Period integrals  \eqref{eq:period-sigma-theta} of a cusp form in $\sigma$ over these two  types of  subgroups of $\Sp_{2n}(\A)$ occur in our computation of Fourier coefficients of  theta lifts of $\sigma$. Their vanishing and non-vanishing essentially determine the first occurrence index. Thus we are  inspired  to investigate distinction by these two types of subgroups. We find that the distinction of residue of $E(g,f_{\chi\otimes\sigma},s)$ and that of cusp forms in $\sigma$ are related. This is a technical result of this paper (Thm.~\ref{thm:dist->period-residue-Eis}), on which several other results are based. 
Since the $L$-function $L(s,\sigma\times\chi)$ appears in the constant term of the
Eisenstein series $E(g,f_{\chi\otimes\sigma},s)$, it is easy to deduce
that the existence of the poles of the partial $L$-function
$L^S(s,\sigma\times\chi)$ implies that of the poles of the Eisenstein
series $E(g,f_{\chi\otimes\sigma},s)$ in the corresponding
location. Thus we get a relation between the
non-vanishing of such periods and the poles of the $L$-functions
$L(s,\sigma\times\chi)$. Hence as discussed in
\cite{jiang14:_autom_integ_trans_class_group_i}, we obtain a
relation between the non-vanishing of such periods attached to $\sigma$
and the occurrence of the simple global Arthur parameter $(\chi,b)$ in
the global Arthur parameter $\psi$ of $\sigma$.
However, it is not known in general whether the existence
the poles of the Eisenstein series $E(g,f_{\chi\otimes\sigma},s)$
implies that the poles of the partial $L$-function
$L^S(s,\sigma\times\chi)$ in the corresponding location. It remains a
hard problem to normalise properly the local intertwining operators
involving cuspidal automorphic representations with non-generic global
Arthur parameters.  



The arguments presented
in this paper are based on two lines of ideas. One is from the work of M\oe
glin (\cite{MR1473165}), which gives a characterisation of the first
occurrence of the irreducible cuspidal automorphic representation
$\sigma$ of $\Sp_{2n}(\BA)$ in the tower of orthogonal groups in terms of
the right-most pole of the Eisenstein series
$E(g,f_{\chi\otimes\sigma},s)$ built from $(\sigma,\chi)$. The other idea is from the work of
\cite{MR2540878} which treats orthogonal groups. Our previous work (\cite{MR3435720}) extends partially their results to the case of unitary groups  and our work in progress \cite{JW-UG-2} will complete the picture for that case.
This paper  extends fully the results to the  case of symplectic groups. In addition, it  exposes more precise relations among the various period integrals.

The main results of this paper are supposed to be parallel to those in \cite{MR2540878,MR3435720,JW-UG-2}. However they are not simple analogues. The family of periods involved are new and special to symplectic groups. First, the periods contain theta kernels to account for various Witt towers of orthogonal groups which form dual reductive pairs with the given symplectic groups. Theta kernels  are only known to be of moderate growth. Much more delicate analysis is employed here to show absolute convergence which is needed to justify change of order in the computation of values of the period integrals. The periods in the unitary case  also call for theta kernels and have their own peculiarities. They will be treated in \cite{JW-UG-2}.  Second, in the symplectic case, we have to consider not only the periods over subgroups of $\Sp_{2n}$ that are symplectic groups, but also those that are Jacobi groups. These new phenomena are the cause for the complexity in the computation of period integrals. 

It is worth mentioning that the periods considered in
\cite{MR2540878} have been used in a recent work of Bergeron, Millson
and M\oe glin on Hodge type theorems for arithmetic manifolds
associated to orthogonal groups (\cite{Bergeron13072016}). It is
expected that similar applications will be found for the periods
studied in \cite{MR3435720} and \cite{JW-UG-2} for unitary groups and
those investigated in this paper for symplectic groups
(\cite{Bergeron2016} and \cite{MR3012154}).

The paper is organised as follows. After setting up some notation in
Sec.~\ref{sec:notation}, we define the Eisenstein series associated to
the data $(\sigma,\chi)$ in Sec.~\ref{sec:pole-eis} and
characterise the locations of their maximal possible poles. We also
show that the maximal pole of the partial $L$-function
$L^S(s,\sigma\times\chi)$ forces the Eisenstein series to have a pole
at a corresponding point (Prop.~\ref{prop:L-pole-Eis-pole}). Then we
define the first occurrences and lowest occurrences of theta lifts
from symplectic groups to orthogonal groups in Sec.~\ref{sec:fo}. The
lowest occurrence is a coarser invariant than the first occurrence since
it is the minimum of first occurrences across many Witt towers. Via
the regularised Siegel-Weil formula, we are able to show in
Thm.~\ref{thm:Eis-pole-LO} that the location of poles of Eisenstein
series provides an upper bound for the lowest occurrence of theta
lifts. This section is largely a restatement of \cite{MR1473165}. It is included for completeness and for fixing notation. Sec.~\ref{sec:periods} contains our main results on periods
(Thms.~\ref{thm:dist->Eis-pole}, \ref{thm:dist->period-residue-Eis}
and Cor.~\ref{thm:FO-Eis-pole}). We get finer results that relate
first occurrences to distinction or non-distinction by symplectic
groups or Jacobi groups of residues of Eisenstein series twisted by
theta functions. One implication
(Thm.~\ref{thm:Eis-pole-LO-strengthened}) is that the inequality in
Thm.~\ref{thm:Eis-pole-LO} is, in fact, an equality. For the periods
we consider, the theta twist is necessary, since the target of theta
lift may not be split orthogonal groups. This is different from the case studied in \cite{MR2540878} where the target of theta lift is symplectic groups (which always split). In addition, in our case both symplectic groups
and Jacobi groups must be considered in order to have a complete picture. The
proofs of our main results occupy the rest of the paper. We indicate
here briefly the line of thoughts. In
Sec.~\ref{sec:fourier-coeff-period}, by studying Fourier coefficients
of theta lifts, we get relations
(Props.~\ref{prop:FO->period-from-fourier-coeff}, \ref{prop:periods-sigma->FO}) between first
occurrences and distinction or non-distinction of automorphic forms in
$\sigma$ twisted by theta functions. The corresponding period
integrals inspire us to consider theta-twisted period integrals of
Eisenstein series of the form in Sec.~\ref{sec:arthur-trunc}. However
due to the fact that those theta-twisted period integrals of
Eisenstein series are generally divergent, we apply the Arthur
truncation method to regularise them. We unfold the period integrals
and analyse the resulting components. Sec.~\ref{sec:value} is devoted
to the computation of values of those components, while
Sec.~\ref{sec:abs-conv-issue} justifies the change of order of
summations and integrations in the computation by dealing with issues
of absolute convergence. We are able to show that after taking residue
at a certain point only one term can possibly remain and that it
contains a theta-twisted period integral on $\sigma$ as inner
integral. Under conditions on distinction of $\sigma$, we can show
that it indeed does not vanish for some choice of data. Thus
distinction of $\sigma$ leads to  distinction of residue of
Eisenstein series. The precise statement can be found in
Thm.~\ref{thm:dist->period-residue-Eis}. Coupled with the result that
conditions on first occurrence supply all the necessary distinction or
non-distinction of $\sigma$, we deduce Cor.~\ref{thm:FO-Eis-pole}. In
this way theta correspondence becomes a means to detect the location
of the maximal pole of the partial $L$-function
$L^S(s, \sigma\times\chi)$, or, in other words, the simple factor of
the form $(\chi,b)$ in the global Arthur parameter $\psi_\sigma$ of
$\sigma$ for largest possible integer $b$. This forms part of the
$(\chi,b)$-theory as explained in \cite{jiang14:_autom_integ_trans_class_group_i}. We plan to
extend the theory to the general $(\tau,b)$-theory for general
irreducible cuspidal automorphic representations $\tau$ of $\GL_a(\BA)$ with
$a\geq 2$.


\section{Notation and the Basic Settings}


\subsection{Notation}
\label{sec:notation}

Let $F$ be a number field and $\A = \A_F$ its
ring of adeles. Fix a non-trivial additive character $\psi_F$ of $F
\lmod \A$.  We will generally write $\psi$ for $\psi_F$ unless there is confusion.
Note that $\psi$ also denotes a global Arthur parameter in the Introduction.  Let
$X$ be a non-degenerate symplectic space over $F$ of even dimension
$m$. The symplectic pairing is denoted by $\form {\ } {\ }_X$. Let
$\cH_a$ denote $a$-copies of the hyperbolic plane. Let
$e_1^+,\ldots,e_a^+,e_1^-,\ldots,e_a^-$ be the basis for $\cH_a$ such
that
$\form {e_i^+} {e_j^-}_{\cH_a} = \delta_{ij} \quad
\text {and} \quad \form {e_i^+} {e_j^+}_{\cH_a} = \form {e_i^-}
{e_j^-}_{\cH_a} = 0$,
for $i,j=1,\ldots, a$, where $\delta_{ij}$ is the
Kronecker delta. Let $\ell_a^+$ (resp. $\ell_a^-$) be the span of
$e_i^+$'s (resp. $e_i^-$'s). Then $\cH_a$ has the polarisation
$\cH_a = \ell_a^+ \oplus \ell_a^-$. Let $X_a = X \perp \cH_a$ be the $(m+2a)$-dimensional
symplectic space. Let $G (X_a)$ be the isometry group which acts on $X_a$ on
the right. Let $Q_a$ be the parabolic subgroup of
$G (X_a)$ that stabilises the isotropic subspace $\ell_a^-$. Let $M_a$
be its Levi subgroup and $N_a$ its unipotent radical. Then
$M_a \isom \GL_a \times G (X)$.
With respect to the `basis' $\ell_a^+, X, \ell_a^-$,
for $x\in \GL_a$ and $h\in G (X)$, we set $m (x,h)$ to be the element
\begin{equation*}
  \begin{pmatrix} x & & \\ & h & \\ && x^*
  \end{pmatrix} \in M_a
\end{equation*} where $x^*\in \GL_a$ is determined by $x$ via the
pairing of $\ell_a^+$ and $\ell_a^-$.  An element in $G (X)$ will be
naturally regarded as an element of $G (X_a)$. Fix a good maximal
compact subgroup $K_a = K_{G (X_a)}$ of $G (X_a) (\A)$ such that the
Iwasawa decomposition holds:
$G (X_a) (\A) = Q_a (\A)K_a$.
For a non-degenerate subspace $Z$ of $X$, set $Z_a$ to
be the subspace $Z \perp \cH_a$ of $X_a$.  Let $Y$ be a non-degenerate
quadratic space over $F$ with symmetric bilinear form $\form {\ } {\
}_Y$. Similarly we form the extended quadratic space $Y_a =Y\perp
\cH'_a$ where $\cH'_a$ is  a split quadratic space as defined analogously to
$\cH_a$. The isometry group
$G (Y_a)$ is considered to act on the left of $Y_a$. As in \cite [(0.7)] {MR1289491}, let $\chi_Y$ be the character of $F^\times \lmod
\A^\times$ given by $\chi_Y (a) = (a,(-1)^{\dim Y (\dim Y -1)/2}\det\form
{\ } {\ }_Y)$,
where $(\ ,\ )$ is the Hilbert symbol.  Thus $\chi_Y =
\chi_{Y_a}$ for all $a \in \Z_{\ge 0}$.

Next we introduce some notation from the {\sl doubling method}. Let
$X'$ be the symplectic space with the same underlying space as $X$ but
with symplectic form
$\form {\ } {\ }_{X'} = - \form {\ } {\ }_{X}$.
We will naturally identify $G (X')$ with $G (X)$. Let
$W$ be the doubled space $X \perp X'$. Then $W$ has the polarisation
$X^{\Delta} \oplus X^{\nabla}$ where
$X^\Delta = \{ (x,x)\in W | x\in X\}$ 
and $X^\nabla = \{ (x,-x)\in W | x\in X\}$.
Also form the extended space $W_a$ and let
$\iota = \iota_1 \times \iota_2 : G (X_a) \times G
(X') \rightarrow G (W_a)$ denote the natural inclusion map. Then 
$W_a = (X^\Delta \oplus \ell_a^+) \oplus (X^\nabla\oplus \ell_a^-)$
is a natural polarisation. Let $P_a$ denote the Siegel parabolic subgroup of $G
(W_a)$ that stabilises $X^\nabla \oplus \ell_a^-$.

\subsection{Poles of Certain Eisenstein Series}
\label{sec:pole-eis}

The results in this
section are not new. They are included for fixing notation and for
completeness. We follow the line of ideas of M\oe glin in \cite{MR1473165} to
determine the poles of a family of Eisenstein series, which will form a basis of arguments in  later sections.  See also \cite{MR3435720}.
We follow mostly the notation of \cite{MR518111}.

Consider the Eisenstein series on
$G (X_a)$ associated to the parabolic subgroup $Q_a$. Since $Q_a$ is a
maximal parabolic subgroup, the space $\fa_{M_a}^*$ is
one-dimensional. We identify $\C$ with $\fa_{M_a,\C}^*$ via 
$s \mapsto s(\frac{m+a+1}{2})^{-1} \rho_{Q_a}$, 
following \cite{MR2683009}, where $\rho_{Q_a}$ is the half sum of the positive
roots in $N_a$. As a result we sometimes regard $\rho_{Q_a}$ as the
number $(m+a+1) /2$.
Let $H_a$ be the homomorphism $M_a (\A) \rightarrow \fa_{M_a}$ such
that for all $m\in M_a (\A)$ and $\xi \in \fa_{M_a}^*$ we have
$\exp (\form { H_a (m)} {\xi}) = \prod_v |\xi (m_v)|_v$.
Explicitly
$\exp (\form { H_a (m (x,h))} {s}) = |\det (x)|^s_\A,
$
for $x\in \GL_a (\A)$ and $h\in G (X) (\A)$. We extend
$H_a$ to a function of $G (X_a) (\A)$ via the Iwasawa decomposition.
Let $\chi$ be a quadratic character of $F^\times \lmod \A^\times$ and
$\sigma \in \cA_\cusp (G (X))$, the set of all equivalence classes of
irreducible cuspidal automorphic representations of $G(X)(\BA)$.
Denote by $\cA_a (s,\chi,\sigma)$ the space of smooth $\C$-valued
functions $f$ on $N_a (\A)M_a (F) \lmod G (X_a) (\A)$ that satisfy the
following properties:
\begin{enumerate}
\item $f$ is right $K_a$-finite;
\item for any $x\in \GL_a (\A)$ and $g\in G (X_a) (\A)$, we have
  \begin{equation*} f (m (x,I_X)g) = \chi (\det (x)) |\det
(x)|_\A^{s+\rho_{Q_a}} f (g);
  \end{equation*}
\item for any fixed $k\in K_a$, the function $h \mapsto f (m (I_a,h)k)$ on $G (X) (\A)$ 
   is a smooth right $K_a \cap G (X) (\A)$-finite
vector in the space of $\sigma$.
\end{enumerate} 
For $f\in \cA_a (0,\chi,\sigma)$ and $s\in \C$, 
a section $f_s (g) := \exp (\form {H_a (g)} {s})f (g)$ of $\cA_a (s,\chi,\sigma)$ 
can be identified with a smooth section in
$\Ind_{Q_a (\A)}^{G (X_a) (\A)} \chi|\ |_\A^s \otimes \sigma$. Form
the Eisenstein series:
\begin{equation*} E^{Q_a} (g,s,f) = \sum_{\gamma \in Q_a (F) \lmod G
(X_a) (F)} f_s (\gamma g).
\end{equation*} By Langlands' theory of Eisenstein series, this is
absolutely convergent for $\Re s > \rho_{Q_a}$ and has meromorphic
continuation to the whole $s$-plane with finitely many poles in the
half plane $\Re s>0$, which are all real in our case. Let $\cP_a
(\chi,\sigma)$ denote the set of positive poles of $E^{Q_a} (g,s,f)$
for $f$ running over $\cA_a (0,\chi,\sigma)$. The number $s_0 >0$ lies
in $\cP_a (\chi,\sigma)$ if and only if for some $f \in \cA_a
(0,\chi,\sigma)$, the Eisenstein series $E^{Q_a} (g,s,f)$ has a pole
at $s=s_0$.

\begin{prop}\label{prop:P1-pole-Pa-pole} Assume that $\cP_1
(\chi,\sigma)$ is non-empty and let $s_0$ be its maximal member. Then
for all integers $a\ge 1$, $s=s_0 + \half (a-1) $ lies in $\cP_a
(\chi,\sigma)$ and is its maximal member.
\end{prop}
\begin{proof} The proof for the unitary case \cite [Prop.~2.1]
{MR3435720} goes through word for word for the current case. See also \cite{MR1473165}.
\end{proof}

\begin{prop}\label{prop:L-pole-Eis-pole} Assume that the partial
$L$-function $L^S (s, \sigma \times \chi)$ has a pole at $s=s_0 >
\half$ and that it is holomorphic for $\Re s> s_0$.  Then for all
integers $a\ge 1$, $s=s_0 + \half (a-1) \in \cP_a (\chi,\sigma)$.
\end{prop}
\begin{proof} By Langlands' theory, the poles of the Eisenstein series is determined by its various constant terms. Since we consider the maximal (i.e., the right-most) pole of the partial $L$-function, it is
enough to consider the poles of the intertwining operator $M
(w_\varnothing,s)$ where $w_\varnothing$ is the longest Weyl element
in $Q_a \lmod G (X_a) / Q_a$. By the Gindikin-Karpelevich formula, we
find the normalising factor to be
  $$\prod_{1\le j \le a} \frac{L^S (s - \half (a-1) +j
-1, \sigma\times \chi)}{L^S (s - \half (a-1) +j , \sigma\times
\chi)}
\cdot \prod_{1\le i < j \le a}\frac{\zeta^S (2s - (a-1) +i +j
-2)}{\zeta^S (2s - (a-1) +i +j -1)}
$$
where $\zeta^S$ are partial Dedekind zeta
functions. This can be simplified to $I_1\cdot I_2$ with
$$
I_1= \frac{L^S (s - \half (a-1),\sigma\times\chi)}{L^S
(s + \half (a+1),\sigma\times\chi)}\quad \text{and}\quad I_2= \prod_{1 < j \le
a}\frac{\zeta^S (2s - (a-1) +j -1)}{\zeta^S (2s - (a-1) +2j -2)}.
$$
There is an extraneous factor in the proof of \cite
[Remarque 2] {MR1473165} which should be absorbed into the
Rankin-Selberg $L$-function.

At $s=s_0 + \half (a-1)$, the numerator of $I_1$ has a pole and it cannot be cancelled out by the factor in the  denominator of $I_1$ or the factors of $I_2$. Thus we
conclude that $s_0 +\half (a-1) \in \cP_a (\chi,\sigma)$.
\end{proof}

Next we relate the Eisenstein series $E^{Q_a} (s,g,f)$ on $G (X_a)$ to
the Siegel Eisenstein series on the `doubled group' $G (W_a)$.

\begin{prop} Let $\F_s$ be a $K_{G (W_a)}$-finite section of $\Ind_{P_a
    (\A)}^{G (W_a) (\A)} \chi |\ |_\A^s$ and $\phi\in \sigma$. Define a
function $F_{\phi,s} (g_a) $ on $G (X_a) (\A)$ by
  \begin{equation}\label{eq:F_phi_s} F_{\phi,s} (g_a) = \int_{G (X)
(\A)} \F_s (\iota (g_a,g))\phi (g) dg.
  \end{equation} Then the following hold.
\begin{enumerate}
\item It is absolutely convergent for $\Re s > (m+a+1)/2$ and has
meromorphic continuation to the whole $s$-plane;
\item It is a section of $\cA_a (s,\chi,\sigma)$;
\item The following identity holds via the meromorphic continuation:
  \begin{equation}\label{eq:EPa-EQa} \int_{[G (X)]} E^{P_a} (\iota
(g_a,g),s,\F_s) \phi (g) dg = E^{Q_a} (g_a,s,F_{\phi,s}).
  \end{equation}
\end{enumerate}
\end{prop}
\begin{proof} Formally we have
  \begin{equation*} F_{\phi,s} (g_a) = \int_{[ G (X)]} \sum_{\gamma
\in G (X) (F)}\F_s (\iota (g_a,\gamma g))\phi (g) dg.
  \end{equation*} Since $\iota_2 (G (X) (F)) \cap P_a (F)$ is trivial
we have an embedding
\begin{equation*} G (X) (F) \hookrightarrow P_a (F) \lmod G (W_a) (F).
\end{equation*} Thus the sum above is a partial sum of an Eisenstein
series, which is absolutely convergent for $\Re s > \rho_{P_a} =
(m+a+1)/2$. Since $\phi$ is cuspidal, the integral defining
$F_{\phi,s}$ is absolutely convergent for $\Re s > (m+a+1)/2$.

Now we check if $F_{\phi,s}(g_a)$ is a section of $\cA_a
(s,\chi,\sigma)$.  Since $\iota_1 ( N_{Q_a} (\A))$ is a subset of
$N_{P_a} (\A)$, $F_{\phi,s}$ is left-invariant under $ N_{Q_a}
(\A)$. For $t_a\in \GL_a (F)$, we have $\iota_1 (m (t_a,1)) \in
M_{P_a} (F)$ and hence $F_{\phi,s}$ is left invariant under the factor
of $M_{Q_a}(F)$ that is isomorphic to $\GL_a (F)$. Let $h$ be an
element in the subgroup $G (X) (\A)$ of $M_{Q_a} (\A)$. Because
$\iota (h,h)\in P_a (\A)$, we have
\begin{align*} F_{\phi,s} (hg_a) =& \int_{G (X) (\A)} \F_s (\iota
(hg_a,g))\phi (g) dg\\ = &\int_{G (X) (\A)} \F_s (\iota (h,h)\iota
(g_a,h^{-1} g))\phi(g) dg\\ = &\int_{G (X) (\A)} \F_s (\iota
(g_a,h^{-1} g))\phi (g) dg\\ =&\int_{G (X) (\A)} \F_s (\iota (g_a,
g))\phi (hg) dg.
\end{align*} Thus $h\mapsto F_{\phi,s} (hg_a)$ is in the space of
$\sigma$.

For meromorphic continuation, it suffices to check that when $h $ belongs to
the subgroup $G (X) (\A)$ of $M_{Q_a} (\A)$, $F_{\phi,s} (h)$ has
meromorphic continuation, or equivalently, that for all $\xi\in \sigma$,
the inner product of $F_{\phi,s}|_{G (X) (\A)}$ and $\xi$ has
meromorphic continuation. By the basic identity in \cite{MR892097}, we
have
\begin{align*} \int_{[G (X)]} F_{\phi,s} (h) \overline {\xi (h)} dh =&
\int_{[G (X)]} \int_{G (X) (\A)} \F_s(\iota (1,g))\phi (hg) \overline
{\xi (h)} dg dh\\ =& \int_{[G (X) \times G (X)] } E^{P} (\iota
(h,g),s,\F_s) \phi (g) \overline {\xi (h)} dg dh
\end{align*} where $E^P$ is an Eisenstein series on $G (W)$ associated
to the Siegel parabolic $P$ and where we regard $\F_s$ as a section of
$\Ind_{P (\A)}^{G (W) (\A)} \chi|\ |_\A^{s+\half a}$ via
restriction. Then meromorphic continuation follows from that of
Eisenstein series.
\end{proof}

We summarise from \cite{MR1174424,MR1411571,MR1159110,MR1289491} the
location of possible poles of the Siegel Eisenstein series.
\begin{prop} The poles with $\Re s>0$ of $E^{P_a} (g,s,F_s)$ are at
most simple and are contained in the set
  \begin{equation*} \Xi_{m+a} (\chi) = \{\half (m+a+1) - j | j\in \Z,
0 < j < \half (m+a+1)\},
  \end{equation*} if $\chi \neq 1$; and in the set
\begin{equation*} \Xi_{m+a} (\chi) = \{\half (m+a+1) - j | j\in \Z, 0
\le j < \half (m+a+1)\}, \end{equation*} if $\chi = 1$.
\end{prop}

We note the following lemmas whose proofs are completely analogous to
those in \cite [Lemma~2.6] {MR3435720}. See also
\cite{MR2540878} and \cite{MR1473165}.
\begin{lemma}  Fix $s_1\in \R$. Then for  large enough $a$, every
section of $\cA_a (s,\chi,\sigma)$ can be approximated by sections of
the form \eqref{eq:F_phi_s} for $\Re s > s_1$.
\end{lemma}
 \begin{lemma} Let $s_0$ be the maximal element in $\cP_1 (\chi,\sigma)$.
   Then for large enough $a$, there exists a section of
$\cA_a (s,\chi,\sigma)$ of the form $F_{\phi,s}$ as in
\eqref{eq:F_phi_s} such that $E^{Q_a} (g,s,F_{\phi,s})$ has a pole at
$s=s_0 + \half (a-1)$.
\end{lemma} This leads to the following:
\begin{prop}\label{prop:pole-in-P1} If $s_0$ is the maximal element in
$\cP_1 (\chi,\sigma)$, then $s_0 = \half (m+2) - j$ for some integer
$j$ with $0 < j < \half (m+2)$.
\end{prop}
\begin{proof} Assume that $s_0$ is the maximal element in $\cP_1
(\chi,\sigma)$, then $s_0 + \half (a-1)$ belongs to $\cP_a
(\chi,\sigma)$ by Prop.~\ref{prop:P1-pole-Pa-pole}. Thus there exists
a section of $\cA_a (s,\chi,\sigma)$ of the form \eqref{eq:F_phi_s}
such that $E^{Q_a} (g,s,F_{\phi,s})$ has a pole at $s=s_0 + \half
(a-1)$. Then from \eqref{eq:EPa-EQa}, this must lie in $\Xi_{m+a}
(\chi)$. Thus $s_0$ must be of the form $\half (m+2) - j$ for some
integer $j$ with $0 \le j < \half (m+2)$.

The Siegel Eisenstein series $E^{P_a}$ can have a pole at $s=\half
(m+a+1)$ only when $\chi=1$. In this case, the space of the residues
of $E^{P_a}$ at $s= \half (m+a+1)$ is the trivial representation. By
\eqref{eq:EPa-EQa}, we see that residue of $ E^{Q_a} (g_a,s,F_{\phi,s})$ must vanish at $s= \half (m+a+1)$. In other words, $s_0$ cannot achieve the value $\half (m+2)$.
\end{proof}

\subsection{First Occurrence of Theta Correspondence}
\label{sec:fo} We recall that $X$ is a symplectic space and that $Y$
is a quadratic space.  We have the space of Schwartz functions on a
maximal isotropic subspace of $(Y\otimes_F X) (\A)$.  When we do not
want to emphasise which maximal isotropic subspace is used, we write
$(Y\otimes X)^+$ for a maximal isotropic subspace and denote by
$\cS_{X,Y} (\A)$ the space of Schwartz functions on $(Y\otimes
X)^+(\BA)$.  The local version is denoted by $\cS_{X,Y} (F_v)$ for a
 place $v$ of $F$.  Since in this paper $Y$ is even dimensional we
do not need to consider the metaplectic cover of $G (X) (\A)$. Recall that we have fixed a non-trivial character $\psi$ of $F\lmod \A$. The
Weil representation $\omega_{\psi,X,Y}$ of $G (X) (\A)\times G (Y)
(\A)$ can be realised on the Schr\"odinger model $\cS_{X,Y} (\A)$. The
explicit formulae can be found, for example, in
\cite{ikeda96:_resid_of_eisen_series_and}.

For $\Phi\in \cS_{X,Y} (\A)$, $g\in G (X) (\A)$ and $h\in G (Y) (\A)$,
we form the theta series
\begin{equation*} \theta_{\psi,X,Y} (g,h,\Phi) := \sum_{w\in (Y\otimes
X)^+ (F)} \omega_{\psi,X,Y} (g,h)\Phi (w).
\end{equation*} Let $\sigma\in \cA_\cusp (G (X))$. For $\phi\in\sigma$
and $\Phi\in \cS_{X,Y} (\A)$, define
\begin{equation}\label{eq:theta-lift-defn} \theta_{\psi,X}^Y
(h,\phi,\Phi) :=\int_{[G (X)]} \theta_{\psi,X,Y} (g,h,\Phi) \phi (g)
dg.
\end{equation} Then the theta lift $\theta_{\psi,X}^Y (\sigma)$ of
$\sigma$ from $G (X)$ to $G (Y)$ is defined to be the span of all such
$\theta_{\psi,X}^Y (h,\phi,\Phi)$'s.  We note that if $Y$ trivial,
then the Weil representation is the one-dimensional trivial representation
and the theta series reduces to a constant function of $G (X) (\A)$.

The first occurrence index $\FO_{\psi}^{Y} (\sigma)$ of $\sigma$ in the Witt
tower of $Y$ is defined to be the smallest number $\dim Y'$ such that
$\theta_{\psi,X}^{Y'} (\sigma)$ is non-zero for $Y'$ running through the
spaces that are in the Witt tower of $Y$.  Fix a quadratic character
$\chi$. Then we define the lowest occurrence index of $\sigma$ with respect
to $\chi$ as follows:
\begin{equation*} \LO_{\psi,\chi} (\sigma) = \min_Y \{\FO_\psi^{Y}
(\sigma)\},
\end{equation*} where $Y$ runs over all non-degenerate quadratic
spaces with $\chi_Y = \chi$ and $\dim Y \equiv 0\pmod {2}$.

With this setup we can state our theorems that derive information on
lowest occurrence of theta lift from poles of Eisenstein series or
$L$-functions.

\begin{thm}\label{thm:Eis-pole-LO} Let $\sigma\in \cA_\cusp (G (X))$
and $s_0$ be the maximal element in the set $\cP_1 (\chi,\sigma)$. Then
  \begin{enumerate}
  \item $s_0 = \half (m+2)-j$ for some integer $j$ such that
    $0 < j < \half (m+2)$;
  \item $\LO_{\psi,\chi} (\sigma) \le 2j$;
  \item $2j \ge r_X$ where $r_X = \frac{m}{2}$ is the Witt index of
$X$.
  \end{enumerate}
\end{thm}
\begin{proof} The first part is just Prop.~\ref{prop:pole-in-P1}. For
the second part we make use of the Siegel-Weil formula.  Most of the
ingredients can be found from \cite{MR1289491}. Part of the theorem
was proved in \cite{MR1473165}. See also \cite{MR3435720}. We give
full details only for the part that is new.

Assume that $s_0=\half (m+2)-j$ is the maximal element in $\cP_1(\chi,\sigma)$.
Then by Prop.~\ref{prop:P1-pole-Pa-pole} and
\eqref{eq:EPa-EQa} we find that $s_0+ \half (a-1) \in \cP_a (\chi,\sigma)$.
Thus there exists a section of $\cA_a (s,\chi,\sigma)$
of the form $F_{\phi,s}$ as in \eqref{eq:F_phi_s} such that $E^{Q_a}
(s,g,F_{\phi,s})$ has a pole at $s=s_0+ \half (a-1)$. Taking $a$ large
so that $\dim Y < m+a+1$, since the residue of the Siegel Eisenstein
series is square-integrable, we get from the Siegel-Weil formula
\cite{MR1289491,MR1863861}:
  \begin{equation*} \res_{s=s_0+\half (a-1)} E^{P_a} (\iota
(g_a,g),s,F_s) = \sum_{\substack{Y: \dim Y =2j\\\chi_Y = \chi}} c_Y
\int_{[G (Y)]} \theta_{\psi} (\iota (g_a,g),h,\omega_{\psi}
(\alpha_Y)\Phi_Y)dh
  \end{equation*} where $\alpha_Y$ is some element in the local Hecke
algebra of $G (W_a)$ at a good place that is used to regularise the
integral, $c_Y$ is some non-zero constant determined by $\alpha_Y$ and
$\Phi_Y$ is some $K_{G (W_a)}$-finite Schwartz function in $\cS_{W_a,Y}
(\A)$.  By integrating both sides over $[G (X)]$ against $\phi (g)$,
we obtain that the left-hand side then becomes the residue at
$s=s_0+\half (a-1)$ of $E^{Q_a} (g_a,s,F_{\phi,s})$. As this is
non-zero, at least one term on the right-hand side is non-zero. In
other words there exists a quadratic space $Y$ of dimension $2j$ and
character $\chi$ such that
\begin{equation*} \int_{[G (X)]} \int_{[G (Y)]} \theta_{\psi} (\iota
(g_a,g),h,\omega_{\psi} (\alpha_Y)\Phi_Y) \phi (g) dh dg\neq 0.
\end{equation*} We may assume that $\omega_{\psi} (\alpha_Y)\Phi_Y =
\Phi_Y^{(1)} \otimes \Phi_Y^{(2)}$ for $\Phi_Y^{(1)} \in \cS_{X_a,Y}
(\A)$ and $\Phi_Y^{(2)} \in \cS_{X,Y} (\A)$. Then after separating
variable we get
\begin{align*} &\int_{[G (X)]} \int_{[G (Y)]} \theta_{\psi,X_a,Y}
(g_a,h,\Phi_Y^{(1)})\theta_{\psi,X,Y} (g,h,\Phi_Y^{(2)})\phi (g)
dhdg\\ =&\int_{[G (Y)]} \theta_{\psi,X_a,Y}
(g_a,h,\Phi_Y^{(1)})\int_{[G (X)]} \theta_{\psi,X,Y}
(g,h,\Phi_Y^{(2)})\phi (g) dg dh.
\end{align*} Thus the inner integral, which is exactly the theta lift
of $\phi \in \sigma$ to $G (Y)$, is non-vanishing. This concludes the
proof of the second part.

For the third part, assume that the lowest occurrence is realised in
the Witt tower of $Y$ so that $\FO_\psi^Y (\sigma) = 2j'$ for some $j'
\le j$. We may assume $\dim Y = 2j'$. Let $\pi = \theta_{\psi,X}^Y
(\sigma)$. Then it is cuspidal as this is the first occurrence. It is
irreducible by \cite [Th\' eor\`eme] {MR1473166}. It is known that
$\FO_{\psi^{-1},Y}^X (\pi) \le 4j'$. By the involutive property in
\cite{MR1473166} we have
$\theta_{\psi^{-1},Y}^X (\pi) =
\theta_{\psi^{-1},Y}^X (\theta_{\psi,X}^Y (\sigma)) = \sigma$.
Thus $m\le 4j' \le 4j$.
\end{proof}
\begin{thm} Let $\sigma\in \cA_\cusp (G (X))$ and $s_0$ be the maximal
element in $\cP_1 (\chi,\sigma)$. Then the following hold.
  \begin{enumerate}
  \item Assume that the partial $L$-function $L^S (s,\sigma\times
\chi)$ has a pole at $s=\half (m+2)-j>0$. Then $\LO_{\psi,\chi}
(\sigma) \le 2j$.
  \item If $\LO_{\psi,\chi} (\sigma) =2j < m+2$, then $L^S
(s,\sigma\times\chi)$ is holomorphic for $\Re s > \half (m+2)-j$.
  \item If $\LO_{\psi,\chi} (\sigma) =2j \ge m+2$, then $L^S
(s,\sigma\times\chi)$ is holomorphic for $\Re s \ge \half $.
  \end{enumerate}
\end{thm}
\begin{proof} This theorem is a consequence of
Thm.~\ref{thm:Eis-pole-LO} and
Prop.~\ref{prop:L-pole-Eis-pole}. Assume that $L^S (s,\sigma\times
\chi)$ has a pole at $s=\half (m+2)-j>0$. Let $\half (m+2)-j'$ be its
maximal pole. Then $s=\half (m+2)-j'$ is a pole of the Eisenstein
series $E^{Q_1}$. Let $\half (m+2)-j''$ be the maximal element of
$\cP_1 (\chi,\sigma)$. We have $j'' \le j' \le j$. By
Thm.~\ref{thm:Eis-pole-LO}, the lowest occurrence $\LO_{\psi,\chi}
(\sigma)\le 2j''\le 2j$.

Assume that $\LO_{\psi,\chi} (\sigma) =2j < m+2$ and $L^S
(s,\sigma\times\chi)$ is not holomorphic for $\Re s > \half
(m+2)-j$. This means that $L^S (s,\sigma\times\chi)$ has a pole at $s
= \half (m+2)-j_0$ for some $j_0 < j$. Part (1) implies that
$\LO_{\psi,\chi} (\sigma) \le 2j_0$ which is strictly less than
$2j$. We get a contradiction.

Assume that $L^S (s,\sigma\times\chi)$ is not holomorphic for $\Re s
\ge \half $. Then $L^S (s,\sigma\times\chi)$ has a pole at $s = \half
(m+2)-j_0$ for some $j_0 \le \half (m+1)$. Then from Part (1), we get
that $\LO_{\psi,\chi} (\sigma) \le 2j_0 \le m+1$.
\end{proof}


\section{Periods and Main Results}
\label{sec:periods}

As noted in the proof of Thm.~\ref{thm:Eis-pole-LO}, for $\sigma\in\CA_\cusp(G(X))$,
the first occurrence of $\sigma$ under the theta correspondence to the
orthogonal group is irreducible and cuspidal. Then by the theory of Li on
the non-singularity of cuspidal automorphic representations
(\cite{MR1168122}), there exists at least one non-zero Fourier
coefficient of maximal rank for the first occurrence representation of
$\sigma$ under the theta correspondence. An explicit calculation of
such a Fourier coefficient leads to our definition of a new family of
period integrals on $\sigma$. Note that the period integrals are
similar to what we defined in \cite{MR3435720}. However they are
different from what was introduced in \cite{MR2540878} due to the presence of theta twist in general.

In this section, for easier indexing, the quadratic space $Y$ is assumed to be anisotropic
so that it sits at the bottom of its Witt tower.

\begin{defn} Let $G$ be a reductive group and $J$ be a subgroup of $G
$. Let $\sigma$ be an automorphic representation of $G$. For
$f\in\sigma$, if the period integral
\begin{equation} \int_{[J]} f (g)dg.
    \end{equation} is absolutely convergent and non-vanishing, we
say that $f$ is $J$-distinguished.  Assume that for all $f\in \sigma$,
the period integral is absolutely convergent.  If there exists $f\in\sigma$ such that $f$ is
$J$-distinguished, then we say $\sigma$ is
$J$-distinguished. If for all $f\in\sigma$,  $f$ is not
$J$-distinguished, then we say $\sigma$ is not
$J$-distinguished.
\end{defn}

\subsection{Fourier Coefficients and Periods}
\label{sec:fourier-coeff-period} 
Let $\Theta_{\psi,X,Y}$ denote the
space of theta functions $\theta_{\psi,X,Y} (\cdot,1,\Phi)$ for $\Phi$
running over $\cS_{X,Y} (\A)$. Let $Z$ be a non-degenerate symplectic subspace of $X$. Then $G(Z)$ is the symplectic group of $Z$.
Let $L$ be a totally isotropic subspace of $Z$. Let $Q(Z,L)$ be the parabolic subgroup of $G(Z)$ that stabilises $L$. Define $J (Z,L)$ to be the Jacobi subgroup of $Q (Z,L)$ that fixes $L$ element-wise. Assume that $L^+$ is any totally isotropic subspace of $Z$ that is dual to $L$ and that a subspace $W \subset Z$ is such that $Z = L^+ \oplus W \oplus L$. Then with respect to the `basis' $L^+, W, L$, we have the following description in terms of  block matrices:
\begin{equation*}
  Q(Z,L) = \left\{\begin{pmatrix}
    * & * & *\\
      & * & *\\
      &   & *
  \end{pmatrix} \in G(Z) \right\},\quad
  J(Z,L) = \left\{\begin{pmatrix}
    I & * & *\\
      & * & *\\
      &   & I
  \end{pmatrix} \in G(Z) \right\}.
\end{equation*}
In our convention the $0$ space is both totally isotropic and anisotropic. We allow $L$ to be the $0$ space.  In this case  $J (Z,L) = G (Z)$. Note that the groups act on the vector space $Z$ from the right.
With this set-up we have the following:

\begin{prop}\label{prop:FO->period-from-fourier-coeff} Let
$\sigma\in\CA_\cusp(G(X))$ and let $Y$ be a (possibly trivial) anisotropic
quadratic space. Assume that
$\FO_{\psi}^{Y} (\sigma) = \dim Y + 2r_0$.
Let $Z$ be a non-degenerate symplectic subspace of
$X$ and $L$ a totally isotropic subspace of $Z$. Let $r$ be a
non-negative integer. Then the following hold.
  \begin{enumerate}
    \item If $\dim X - \dim Z + \dim L + r < r_0$, then $\sigma
\otimes \Theta_{\psi,X,Y_r}$ is not $J (Z,L)$-distinguished for any choice of $(Z,L)$.
    \item If $\dim L =0, 1$ and $\dim X - \dim Z + \dim L= r_0$, then
$\sigma \otimes \Theta_{\psi,X,Y}$ is $J (Z,L)$-distinguished for some choice of $(Z,L)$.
  \end{enumerate}
\end{prop}

\begin{proof} We will suppress the subscript $\psi$ and taking $F$-rational
points unless there is confusion. For example $G (X) (F)$ will simply be $G (X)$. Let $\phi\in\sigma$ and $\Phi\in\cS_{X,Y_r} (\A)$. The periods under consideration are
  \begin{equation}\label{eq:period-sigma-theta}
    \int_{[J (Z,L)]}\phi (g) \theta_{X,Y_r} (g,1,\Phi) dg.
  \end{equation}
Since $\phi$ is a cuspidal automorphic form and the
theta series is of moderate growth, the whole integrand is rapidly
decreasing on the Siegel domain of $G (X)$. The Jacobi group $J (Z,L)$
is a semi-direct product of some symplectic subgroup $G (Z')$ of $G
(X)$ and a unipotent group. It can be seen that the Siegel domain of
$G (Z')$ is contained some Siegel domain of $G (X)$. Thus the period
integrals are all absolutely convergent.

 Recall that $\theta_X^{Y_r} (h, \phi, \Phi)$ denotes the theta
lift in \eqref{eq:theta-lift-defn} for $\phi\in \sigma$ and $\Phi\in
\cS_{X,Y_r} (\A)$. We consider its $\beta$-th Fourier coefficients
$\theta_{X,\beta}^{Y_r}(h, \phi, \Phi)$ for $\beta \in \Alt_a $ for
$a\le r$ of the form
$\smatrix{\beta_0}{0}{0}{0} $,
where  $\beta_0\in\Alt_b$ is a non-degenerate element. Note that $b$
must be even. Let $n_a$ denote the homomorphism:
\begin{equation*}
  \Alt_a \rightarrow G (Y_r), \ \ \ \
  z \mapsto
  \begin{pmatrix} I_a&0&z\\ &I &0 \\ && I_a
  \end{pmatrix}.
\end{equation*}
Then we have
\begin{equation}\label{eq:beta-Fourier-coeff-of-theta-lift}
\theta_{X,\beta}^{Y_r} (h,\phi,\Phi) = \int_{[\Alt_a]}\int_{[G (X)]}
\phi (g)\theta_X^{Y_r} (g,n_a (z)h,\Phi)\overline{\psi (\tr (\beta
z))}dgdz .
\end{equation}
In order to compute $\theta_X^{Y_r} (g,n_a (z)h,\Phi)$ we
realise the Weil representation on $\cS ((
\ell_a^-\otimes X \oplus Y_{r-a}\otimes X^+) (\A))$. Then
\begin{align*} \theta_X^{Y_r} (g,n_a (z)h,\Phi) =& \sum_{u,v,w}
\omega_{X,Y_r} (g,n_a (z)h)\Phi (u,v,w)
\end{align*} where the summation runs over $u\in \ell_a^- \otimes
X^+$, $v\in \ell_a^- \otimes X^- $ and $w\in Y_{r-a} \otimes
X^+$. Writing out the action explicitly, we find that each term is equal to
\begin{align*} \psi (\tr (zu \trpz {v}))\omega_{X,Y_r} (g,h)\Phi
(u,v,w).
\end{align*} Thus integration over $dz$ in
\eqref{eq:beta-Fourier-coeff-of-theta-lift} vanishes unless $u\trpz
{v} = \beta$, which implies that
$\form {x} {x}_X = x
  \begin{pmatrix} & I\\-I &
  \end{pmatrix}\trpz {x} = 2\beta$, 
if we set $x= (u, v) \in \ell_a^- \otimes X$. Thus
\eqref{eq:beta-Fourier-coeff-of-theta-lift} is equal to
\begin{equation*} \int_{[G (X)]} \sum_{\substack{x,w\\ \form {x} {x}_X
= 2\beta}} \phi (g)\omega_{X,Y_r} (g,h)\Phi (x,w) dg
\end{equation*} where $x$ runs over $\ell_a^- \otimes X$ and $w$ over
$ Y_{r-a} \otimes X^+$. Write
$
x = \bigl(
\begin{smallmatrix}
  x_0 \\ x_1
\end{smallmatrix}
\bigr)
$
for $x_0 \in \ell_b^- \otimes X$ and $x_1\in
\ell_{a-b}^- \otimes X$. Thus we need $\form {x_0} {x_0}_X = \beta_0$,
$\form {x_0} {x_1}_X=0$ and $\form {x_1} {x_1}_X=0$. We note that
$x_0$'s form one orbit under action of $G (X)$, that $\Span\{x_1\}$ is
a totally isotropic subspace contained in the orthogonal complement of
$\Span\{x_0\}$ and that the $G (X)$-orbits of $\Span\{x_1\}$'s are
parametrised by dimension. We fix one $x_0 \in \ell_b^-\otimes X$ such
that $\form {x_0} {x_0}_X = \beta_0$ and $x_1^{(i)} \in
\ell_{a-b}^-\otimes X$ of rank $i$ such that $\form {x_0}
{x_1^{(i)}}_X=0$ and $\form {x_1 ^{(i)}} {x_1^{(i)}}_X=0$ for $i=0,\ldots,
\min\{(\dim X - b)/2, a-b\}$. To parametrise $x_1$ rather than
$\Span\{x_1\}$ we need to apply $\GL_{a-b}$-action on the left. The
stabiliser in $G (X)$ of $x_0$ and $\Span\{x_1^{i}\}$ is $J
(\{x_0\}^{\perp}, x_1^{(i)})$ where $\{x_0\}^{\perp}$ denotes the
orthogonal complement of $\Span\{x_0\}$ in $X$ and to avoid clutter we
write $x_1^{(i)}$ rather than $\Span\{x_1^{(i)}\}$ here. Thus
\eqref{eq:beta-Fourier-coeff-of-theta-lift} is equal to
\begin{align*} &\int_{[G (X)]} \sum_i \sum_\delta \sum_\gamma
\sum_{w}\phi (g)\omega_{X,Y_r} (g,h)\Phi (
 \begin{pmatrix}
x_0 \\ \delta^{-1}x_1^{(i)}
 \end{pmatrix}
\gamma, w)dg \\ =& \int_{[G (X)]} \sum_i \sum_\delta \sum_\gamma
\sum_w \phi (g)\omega_{X,Y_r} (\gamma g,m_{G (Y_r),a-b} (\delta)
h)\Phi (
  \begin{pmatrix} x_0 \\ x_1^{(i)}
  \end{pmatrix}, w)dg\\
\end{align*} where $w$ runs over $Y_{r-a}\otimes X^+$, $i$ over
$\{0,\ldots, \min\{(\dim X - b)/2, a-b\}\}$, $\delta$ over
$(\Stab_{\GL_{a-b}} x_1^{(i)}G (X)) \lmod \GL_{a-b}$ and $\gamma$ over
$J (\{x_0\}^{\perp}, x_1^{(i)}) \lmod G (X)$. Here $\Stab_{\GL_{a-b}}
x_1^{(i)}G (X)$ denotes the stabiliser in $\GL_{a-b}$ of the $G
(X)$-orbit of $x_1^{(i)}$ and we set
\begin{equation*}
m_{G (Y_r),a-b} : \GL_{a-b} \rightarrow G (Y_r),\ \ \ \
\delta \mapsto\diag(I_b,\delta,I,\delta^*,I_b)
\end{equation*}
with $\delta^*$ determined by $\delta$ via the quadratic form
on $Y_r$. We may freely change order of summations and integration
because of absolute convergence. We collapse the sum over $\gamma$ and
the integration to get
\begin{align} & \sum_i \sum_\delta \int_{J (\{x_0\}^{\perp},
x_1^{(i)}) \lmod G (X) (\A)} \phi (g) \sum_w \omega_{X,Y_r} ( g,m_{G
(Y_r),a-b} (\delta) h)\Phi (
   \begin{pmatrix} x_0 \\ x_1^{(i)}
 \end{pmatrix}, w) dg\nonumber\\
\label{eq:sum-of-periods} =& \sum_i \sum_\delta \int_{J
 (\{x_0\}^{\perp}, x_1^{(i)}) (\A) \lmod G (X) (\A)} \int_{[J
(\{x_0\}^{\perp}, x_1^{(i)}) ]} \phi (g' g)\\ &\qquad
\theta_{X,Y_{r-a}} ( g' , 1, \omega_{X,Y_r} (g, m_{G (Y_r),a-b}
(\delta) h)\Phi (
    \begin{pmatrix} x_0 \\ x_1^{(i)}
    \end{pmatrix}, \cdot))dg' dg.\nonumber
\end{align}

We separate the discussion for the case of $r<r_0$ and the case of $r=r_0$.

First, suppose $r<r_0$. Then \eqref{eq:sum-of-periods} is identically
$0$, because it is a Fourier coefficient of an earlier occurrence than the
first occurrence. We use induction on $\dim L$ to show part (1). Assume that $b=a\le
r$. In this case $a$ is forced to be even. Then \eqref{eq:sum-of-periods} being equal to $0$ simply means
\begin{align*} 0=& \int_{J (\{x_0\}^{\perp}, 0) (\A)\lmod G (X) (\A)}
\int_{[J (\{x_0\}^{\perp}, 0)]} \phi (g' g) \theta_{X,Y_{r-a}} (g'
,1, \omega_{X,Y_r} ( g, h)\Phi ( x_0 , \cdot))dg' dg.
\end{align*} 

Consider the inner integral 
$\int_{[J (\{x_0\}^{\perp}, 0)]} \phi (g' )
\theta_{X,Y_{r-a}} (g' ,1, \Psi)dg'$, 
where $\Psi \in \cS_{X,Y_{r-a}} (\A)$. For all choice
of data, the above must vanish. Otherwise we can choose $\Psi' \in \cS
(\ell_a^-\otimes X (\A))$ supported very close to $x_0$ at ramified places such that
\begin{multline*} \int_{J (\{x_0\}^{\perp}, 0) (\A)\lmod G (X) (\A)}
\int_{[J (\{x_0\}^{\perp}, 0)]} \phi (g' g) \\\theta_{X,Y_{r-a}} (g'g
,1, \omega_{X,\ell_a^+\oplus \ell_a^-} (g,1)\Psi' (x_0) \Psi (\cdot))dg' dg
\end{multline*} is non-vanishing, which is a contradiction. Letting $a$ and
$\beta_0$ vary as long as $b=a\le r < r_0$, we get that
$\int_{[J (Z, L)]} \phi (g' ) \theta_{X,Y_t} (g' ,1,
\Psi)dg' = 0$, 
for all choice of data as long as $\dim L=0$ and $\dim
X -\dim Z +\dim L +t <r_0$.

Assume that
$\int_{[J (Z, L)]} \phi (g' ) \theta_{X,Y_t} (g' ,1,
\Psi)dg' = 0$, 
for all choice of data whenever $\dim L < n_0$ and
$\dim X -\dim Z +\dim L +t <r_0$.  We consider the case when $a-b =
n_0$. Then those terms in \eqref{eq:sum-of-periods} corresponding to
$i < n_0$ must vanish by the induction hypothesis. Also the set
$\Stab_{\GL_{a-b}} x_1^{(n_0)}G (X) \lmod \GL_{a-b}$ that $\delta$ runs
over is trivial. Thus there is only one term left:
\begin{multline*} \int_{J (\{x_0\}^{\perp}, x_1^{(a-b)}) (\A) \lmod G
(X) (\A)} \int_{[J (\{x_0\}^{\perp}, x_1^{(a-b)}) ]} \phi (g' g)\\
\theta_{X,Y_{r-a}} ( g' , 1, \omega_{X,Y_r} (g, h)\Phi (
    \begin{pmatrix} x_0 \\ x_1^{(a-b)}
    \end{pmatrix}, \cdot))dg' dg.
\end{multline*} By a similar argument as in the case $a=b$, the inner
integral must vanish.  Thus for all choice of data, the integral
$\int_{[J (Z,L) ]} \phi (g') \theta_{X,Y_t} ( g', 1,
\Psi)dg'$
must vanish whenever $\dim L = n_0$ and $\dim X -\dim
Z +\dim L +t <r_0$. We conclude the proof for part (1).

Now we consider the case  $r=r_0$ to prove part (2). In this case the theta lift
does not vanish for some choice of $\phi\in\sigma$ and
$\Phi\in\cS_{X,Y_r}(\A)$. We take $\beta$-th Fourier coefficients for $\beta\in \Alt_{r_0}$. In other words, $a=r_0$. By the Main Theorem of \cite{MR1168122}
there exists $\beta$ of maximal possible rank such that the $\beta$-th
Fourier coefficient does not vanish. When $r_0$ is even, $\beta$ is of
rank $r_0$. There is only one term left in \eqref{eq:sum-of-periods}
corresponding to $i=0$:
\begin{equation}\label{eq:i=0-term-in-fourier-coeff-of-theta-lift}
  \int_{J (\{x_0\}^{\perp}, 0) (\A) \lmod G (X) (\A)} \int_{[J
(\{x_0\}^{\perp}, 0) ]} \phi (g' g) \theta_{X,Y} ( g' , 1,
\omega_{X,Y_r} (g, h)\Phi (x_0, \cdot))dg' dg
\end{equation}
and it does not vanish.
Note that there is no summation over $\delta$ and  that the inner integral is of the  form: $\int_{[J (Z,0) ]} \phi (g') \theta_{X,Y} ( g',
1,\Psi) dg'$ for some choice of $\phi$ and $\Psi$. As  \eqref{eq:i=0-term-in-fourier-coeff-of-theta-lift} does not vanish, we
conclude that $\int_{[J (Z,0) ]} \phi (g') \theta_{X,Y} ( g',
1,\Psi) dg'$ 
does not vanish for some choice of data.
When $r_0$ is odd, $\beta$ is of rank $r_0-1$. There
are only two terms left in \eqref{eq:sum-of-periods} corresponding to $i=0$ and
$1$:
\begin{align}\label{eq:i=1-term-in-fourier-coeff-of-theta-lift}
  &\int_{J (\{x_0\}^{\perp}, 0) (\A) \lmod G (X) (\A)} \int_{[J
(\{x_0\}^{\perp}, 0) ]} \phi (g' g) \theta_{X,Y} ( g' , 1,
    \omega_{X,Y_r} (g, h)\Phi (x_0, \cdot))dg' dg\\
+& \int_{J (\{x_0\}^{\perp}, x_1^{(1)}) (\A) \lmod G (X) (\A)} \int_{[J
   (\{x_0\}^{\perp}, x_1^{(1)}) ]} \phi (g' g)\nonumber\\
  &\qquad\qquad\qquad\theta_{X,Y} ( g' , 1,
\omega_{X,Y_r} (g, h)\Phi (
    \begin{pmatrix} x_0 \\ x_1^{(1)}
    \end{pmatrix}, \cdot))dg' dg
  \nonumber
\end{align}
and this sum does not vanish. For the term corresponding to $i=1$, in fact, there should be a summation over $\delta \in \Stab_{\GL_{1}} x_1^{(1)}G (X) \lmod \GL_{1}$, but one can check that this set has only one element which is the identity element.
Now we look at the two terms.
 The first has an inner integral of the form
$\int_{[J (Z,0) ]} \phi (g') \theta_{X,Y} ( g',
1,\Psi) dg'$
with $\dim X - \dim Z = r_0 -1 < r_0$ and thus must
vanish by part (1). Thus the term corresponding to $i=1$ must be
non-vanishing. This has an inner integral of the form
$\int_{[J (Z,L) ]} \phi (g') \theta_{X,Y} ( g', 1,
\Psi)dg'$ with $\dim L =1$ and $\dim X - \dim Z + \dim L =
r_0$. This must be non-vanishing for some choice of data.
\end{proof}

\begin{rmk} If $Y= 0$, the theta series reduces to a
constant. Compare with Prop.~5.2 in \cite{MR2540878}.  When $Z=X$ and $L=0$, the
period integral is just the integral for theta lift of $\sigma$ to $G
(Y_r)$.   In fact the above period integrals can be replaced by
$\int_{[J (Z,L)]} \phi (g') \theta_{\psi,Z,Y}
(g',1,\Psi) dg'$, where $\Psi \in\cS_{Z,Y} (\A)$.
\end{rmk}

There is a certain converse to the previous proposition.
\begin{prop} \label{prop:periods-sigma->FO} Let $\sigma \in \cA_\cusp (G (X))$ and $Y$ be a (possibly trivial) anisotropic quadratic space.
Assume that $\sigma \otimes \Theta_{\psi,X,Y}$ is $J(Z,L)$-distinguished for some $(Z,L)$ with $\dim L =0$ or $1$, but that it is not $J (Z',L')$-distinguished for
all $(Z',L')$ with $\dim L'=0$ or $1$ and $\dim Z' - \dim L' > \dim Z - \dim L$. Let $r=
\dim X- \dim Z +\dim L$. Then $\FO_\psi^Y (\sigma) = \dim Y + 2r$.
\end{prop}
\begin{proof} We will drop the subscript $\psi$ to avoid clutter. Let
$r_0 = (\FO_\psi^Y (\sigma) - \dim Y)/2$. We need to show $r_0 = r$.

Suppose first that $L$ is the zero space. The period integral
corresponding to $J (Z,0)$-distinction is 
$\int_{[J (Z,0) ]} \phi (g') \theta_{X,Y} ( g',
1,\Psi) dg'$.
It occurs as an inner integral of the $\beta$-th Fourier
coefficient of  $\theta_{X}^{Y_r} (h,\phi,\Phi)$ in the space of theta lift of
$\sigma$ to $G (Y_r)$
(c.f. \eqref{eq:i=0-term-in-fourier-coeff-of-theta-lift}) for some
$\beta \in \Alt_r$ of full rank $r$. As in the proof of Prop.~\ref{prop:FO->period-from-fourier-coeff}, for some choice of data, the Fourier coefficient is non-vanishing. Thus $r_0 \le r$.

Suppose second that $L$ is an isotropic line. The period integral
corresponding to $J (Z,L)$-distinction is 
$\int_{[J (Z,L) ]} \phi (g') \theta_{X,Y} ( g',
1,\Psi) dg'$.
This is not directly an inner integral of a certain
Fourier coefficient of a theta lift. Consider the $\beta$-th Fourier
coefficient of theta lift of $\sigma$ to $G (Y_{r})$ where $\beta\in\Alt_r$ is
of the form
$\smatrix{\beta_0}{}{}{0} $
where $\beta_0$ is non-degenerate in
$\Alt_{r-1}$. From \eqref{eq:i=1-term-in-fourier-coeff-of-theta-lift},
we see that our period integral is an inner integral of the second
double integral for some choice of data. The inner integral of the
first double integral is of the form
\begin{equation*} \int_{[J (Z',0) ]} \phi (g') \theta_{X,Y} ( g',
1,\Psi) dg'
  \end{equation*} with $\dim Z' = \dim X - (r -1 )$. Thus $\dim Z'
> \dim Z - \dim L$, so the first term in
\eqref{eq:i=1-term-in-fourier-coeff-of-theta-lift} vanishes by
assumption. The $J (Z,L)$-distinction means that the inner
integral of the second term in
\eqref{eq:i=1-term-in-fourier-coeff-of-theta-lift} is non-vanishing
for some choice of data. By a similar argument as in
Prop.~\ref{prop:FO->period-from-fourier-coeff}, we can further choose
some more data to make the outer integral non-vanishing. This shows
that $r_0 \le r$.

We need to show equality $r_0 =r$. Assume that $r_0 < r $. Then
Prop.~\ref{prop:FO->period-from-fourier-coeff} shows that
$\int_{[J (Z',L') ]} \phi (g') \theta_{X,Y} ( g',
1,\Psi) dg'$
is non-vanishing for some choice of data with $\dim
X - \dim Z' + \dim L' = r_0$. Here $\dim L' = 0$ or $1$ is determined
by parity of $r_0$. Then $\dim Z' - \dim L'=\dim X - r_0 > \dim X -r =
\dim Z -\dim L$, which contradicts our assumption.
\end{proof}

We also get  the following implication.
\begin{prop}
Let $\sigma \in \cA_\cusp (G (X))$ and $Y$ be a (possibly trivial) quadratic space.
Assume that for some $t>0$, $\sigma \otimes \Theta_{\psi,X,Y}$ is not $J (Z',L')$-distinguished for
all $(Z',L')$ with $\dim L'=0$ or $1$ and $\dim Z' - \dim L' > t$. Then $\sigma \otimes \Theta_{\psi,X,Y}$ is not
  $J (Z'',L'')$-distinguished for
all $(Z'',L'')$ with $\dim Z'' - \dim L'' > t$.
\end{prop}
\begin{proof}
Assume that $\FO_\psi^Y (\sigma) = \dim Y + 2 (\dim X -t_0)$.
Then $\sigma \otimes \Theta_{\psi,X,Y}$ is $J (Z,L)$-distinguished for
some $(Z,L)$ with $\dim L=0$ or $1$ and $\dim Z - \dim L = t_0$. Thus $t_0 \le t$. Prop.~\ref{prop:FO->period-from-fourier-coeff} shows that for all $(Z'',L'')$ with  $\dim Z'' - \dim L'' > t_0$, $\sigma \otimes \Theta_{\psi,X,Y}$ is not $J (Z'',L'')$-distinguished. Thus a fortiori, for all $(Z'',L'')$ with  $\dim Z'' - \dim L'' > t$, $\sigma \otimes \Theta_{\psi,X,Y}$ is not $J (Z'',L'')$-distinguished.
\end{proof}


\subsection{The Arthur Truncation Method}
\label{sec:arthur-trunc}

In this section we still assume that $Y$ is anisotropic. Let $Z$ be a
non-degenerate subspace of $X$. Let $V$ denote its orthogonal complement in $X$. Similarly we form the extended space
$Z_1 = \ell_1^+ \oplus Z \oplus \ell_1^-$. Thus $Z_1$ is a
non-degenerate subspace of $X_1$. Let $L$ be a totally isotropic space of $Z$ of dimension $0$ or $1$.
We consider period integrals of the
following form
\begin{equation*} \int_{[J(Z_1, L)]} E^{Q_1} (g,s,f_s)
\theta_{\psi,X_1,Y} (g, 1, \Phi) dg
\end{equation*} for $f_s \in \cA_1 (s,\chi,\sigma)$ and $\Phi \in
\cS_{X_1,Y} (\A)$. The integrals diverge in general, and can be regularised by applying
Arthur truncation\cite{MR558260,MR518111} to the Eisenstein series, for instance. We
will drop the superscript $Q_1$ from now on.

Since the parabolic group $Q_1$ is maximal, $\fa_{M_1}$ is
one-dimensional and is identified with $\R$. Define two functions on
$\R$ as follows. For $c\in \R_{>0}$, set $\hat {\tau}^c$ to be the
characteristic function of $\R_{> \log c}$ and set $\hat {\tau}_c =
1_{\R} - \hat {\tau}^c$. We put in $\log$ to ease later computation.
Then in our special case the Arthur truncation formula has only two
terms
\begin{equation*} \Lambda^c E (g,s,f_s) = E (g,s,f_s) - \sum_{\gamma
\in Q_1 \lmod G (X_1)} E_{Q_1} (\gamma g,s,f_s) \hat {\tau}^c (H
(\gamma g))
\end{equation*} where $E_{Q_1}$ denotes the constant term of the
Eisenstein series along $Q_1$. The summation has only finitely many
terms (depending on $g$).
For $\Re s > \rho_{Q_1}$ we
have
\begin{equation*} E_{Q_1} (g,s,f_s) = f_s (g) + M (w,s)f_s (g),
\end{equation*} where $w $ is the longest Weyl element in $Q_1 \lmod G
(X_1) / Q_1$.  The above identity actually holds for all $s$ as
meromorphic functions.
Thus the truncated Eisenstein series $\Lambda^c E (g,s,f_s)$ equals
\begin{equation*} \sum_{\gamma \in Q_1 \lmod G (X_1)} f_s (\gamma g)
\hat {\tau}_c (H (\gamma g)) - \sum_{\gamma \in Q_1 \lmod G (X_1)} M
(w,s) f_s (\gamma g) \hat {\tau}^c (H (\gamma g)),
\end{equation*} where the first summation is absolutely convergent for
$\Re s > \rho_{Q_1}$ and the second has only finitely many terms. Both
have meromorphic continuation to the whole complex plane.

Set
\begin{align*} \xi_{c,s} (g) = & f_s (g)\theta_{\psi,X_1,Y} (g, 1,
\Phi)\hat {\tau}_c (H (g)) ;\\ \xi^c_s (g) = & M (w,s) f_s
(g)\theta_{\psi,X_1,Y} (g, 1, \Phi)\hat {\tau}^c (H (g)) .
\end{align*} Also set
\begin{align}\label{eq:I-xi-s} I (\xi) =&\int_{[J (Z_1,L)]}
\sum_{\gamma \in Q_1 \lmod G (X_1)} \xi (\gamma g) dg
\end{align} for $\xi= \xi_{c, s}$ or $\xi^c_s$.  Then
\begin{equation*} \int_{[J (Z_1,L)]} \Lambda^c E (g,s,f_s)
\theta_{\psi,X_1,Y} (g, 1, \Phi) dg = I (\xi_{c, s}) - I (\xi^c_s).
\end{equation*} Since the truncated Eisenstein series is rapidly
decreasing, the left-hand side is absolutely convergent. The absolute
convergence of the two terms on the right will be discussed later.

We examine the $G (Z_1)$-orbits in the generalised flag variety $Q_1
\lmod G (X_1)$ which classifies isotropic lines in $X_1$. We use
the correspondence given below:
\begin{align*} Q_1 \lmod G (X_1) &\longleftrightarrow \{\text
{Isotropic lines in } X_1\}\\ \gamma &\longleftrightarrow
\ell_1^-\gamma.
\end{align*} Given an isotropic line $\ell$ in $X_1$ we write
$\gamma_\ell$ for an element in $ G (X_1)$ such that $\ell_1^-
\gamma_\ell = \ell$. In the following we always try to choose
$\gamma_\ell$ as simple as possible to facilitate computation.

 Given an isotropic line in $X_1$ we pick a non-zero vector $x$ on the
line and write $x = v + z$ according to the decomposition
$X_1 = V  \perp Z_1$. We define
\begin{enumerate}
\item $\Omega_{0,1}$ to be set of (isotropic) lines in $X$ whose
projection to $V$ is $0$,
\item $\Omega_{1,0}$ to be set of (isotropic) lines in $X$ whose
projection to $Z_1$ is $0$,
\item $\Omega_{1,1}$ to be set of the remaining (isotropic) lines.
\end{enumerate} Note that lines in $X$ are automatically isotropic as
they lie in a symplectic space and that each set is stable under the
action of $G (Z_1)$.

We analyse the  $G(Z_1)$-orbits of each of these sets.
The set $\Omega_{0,1}$ forms one $G (Z_1)$-orbit. We take $Fe_1^- =
\ell_1^-$ to be the orbit representative. Its stabiliser in $G (Z_1)$
is the parabolic subgroup $Q (Z_1, Fe_1^-)$ that stabilises
$Fe_1^-$. The set $\Omega_{1,0}$ has orbit representatives $Fv$ for
$v$ running over a set of representatives of $V-\{0\} / F^\times$. The
stabiliser of $Fv$ is $G (Z_1)$. Consider the set $\Omega_{1,1}$. For
$v\in V - \{0\}$, the line $F (v + z)$ is in the same orbit as $F
(v+e_1^-)$. Let $v_1, v_2 \in V - \{0\}$. For $F (v_1 + e_1^-)$ to be
in the same orbit as $F (v_2 + e_1^-)$ there should exist $\gamma \in G (Z_1)$
and $c\in F^\times$ such that
\begin{equation*} v_1 + e_1^-\gamma = c (v_2 + e_1^-).
\end{equation*} Thus $F (v+e_1^-)$ gives a set of orbit
representatives as $v$ runs over a set of representatives of $V-\{0\}
/ F^\times$. The stabiliser is $J (Z_1, Fe_1^-)$.
Thus we write
\begin{equation*} I (\xi) = I_{\Omega_{0,1}} (\xi) +I_{\Omega_{1,0}}
(\xi) +I_{\Omega_{1,1}} (\xi)
\end{equation*} with
\begin{align}
\label{eq:I-Omega01} I_{\Omega_{0,1}} (\xi) &= \int_{[J (Z_1,L)]}
\sum_{\delta\in Q (Z_1,Fe_1^-) \lmod G (Z_1)} \xi (\delta g) dg;\\
\label{eq:I-Omega10} I_{\Omega_{1,0}} (\xi) &=\int_{[J (Z_1,L)]}
\sum_{v\in V-\{0\} / F^\times} \xi (\gamma_{Fv} g)dg;\\
\label{eq:I-Omega11} I_{\Omega_{1,1}} (\xi) &=\int_{[J (Z_1,L)]}
\sum_{v\in V-\{0\} / F^\times} \sum_{\delta\in J (Z_1,Fe_1^-) \lmod G
(Z_1)} \xi (\gamma_{F (v+e_1^-)}\delta g)dg.
\end{align}

\begin{prop}\label{prop:abs-conv}
  For $\xi = \xi_{c,s}$ or $\xi_s^c$, the integrals $I_{\Omega_{0,1}} (\xi)$, $I_{\Omega_{1,0}} (\xi)$ and $I_{\Omega_{1,1}} (\xi)$ are absolutely convergent for $\Re s$ and the truncation parameter $c$ large enough.
\end{prop}
This will be proved  in Sec.~\ref{sec:abs-conv-issue} where more precise statements will be given. For some of them we can relax  conditions on $s$ or $c$.

\subsection{Periods of Truncated Eisenstein Series and Residues}

\label{sec:FO->pole}

First we state our results on the relation between the distinction and certain
theta-twisted periods of residue of Eisenstein series. The proofs hinge
on the computation (c.f. Sec.~\ref{sec:value}) of various integrals occurring in the period integrals.

To avoid too much repetition, in this section, $Z$ (resp. $Z'$, $Z''$) is
always a non-degenerate subspace of $X$ and $L$ (resp. $L'$, $L''$) is always
a totally isotropic subspace of $Z$ (resp. $Z'$, $Z''$).
\begin{thm} \label{thm:dist->Eis-pole}
  Let $\sigma \in \cA_\cusp (G (X))$ and let $Y$ be a
(possibly trivial) anisotropic quadratic space.  Assume that
$\sigma\otimes\Theta_{\psi,X,Y}$ is $J (Z,L)$-distinguished for some $(Z,L)$ such that $\dim
L =0$ or $1$, but that it is not $J (Z',L')$-distinguished for all $(Z',L')$ such that $\dim
L' = 0, \ldots, \dim L +1$ and $\dim Z' - \dim L' > \dim Z - \dim
L$. Let $r = \dim X - \dim Z +\dim L$ and $s_0=\half (\dim X - (\dim Y
+2r) +2)$. Then
$s_0$ is a  pole of $E^{Q_1}
(s,g,f_s)$ for some $f_s\in \cA_1 (s,\chi_{Y},\sigma)$.
\end{thm}
\begin{proof}
We claim that the intertwining operator $M (w,s)$ must have a pole at
$s=s_0$. Indeed, assume otherwise. Consider the period integral
\begin{equation*}
  \int_{[J (Z_1,L)]} \Lambda^c E (g,s,f_s)
\theta_{\psi,X_1,Y} (g, 1, \Phi) dg.
\end{equation*}
In Sec.~\ref{sec:arthur-trunc} we have shown that it is equal to
\begin{equation*}
(  I_{\Omega_{0,1}} (\xi_{c,s}) + I_{\Omega_{1,0}} (\xi_{c,s}) + I_{\Omega_{1,1}} (\xi_{c,s})) - (  I_{\Omega_{0,1}} (\xi_s^c) + I_{\Omega_{1,0}} (\xi_s^c) + I_{\Omega_{1,1}} (\xi_s^c)).
\end{equation*}

Let $\xi=\xi_{c,s}$ or $\xi_s^c$. First assume that $\dim L=0$. Then Propositions~\ref{prop:van-I-Omega-10}, \ref{prop:van-I-Omega-11} show that $I_{\Omega_{1,0}}(\xi)$ and $I_{\Omega_{1,1}}(\xi)$ vanish. As we assume that  $M (w,s)$ does not have a pole at
$s=s_0$, we see from \eqref{eq:I-Omega01-computed-expression} that $I_{\Omega_{0,1}} (\xi_s^c)$ does not contribute a pole at $s=s_0$.  Proposition~\ref{prop:non-van-of-key-period-L=0} shows that $I_{\Omega_{0,1}} (\xi_{c,s})$ has a pole at $s=s_0$ for some choice of data. Thus the truncated Eisenstein series and also the Eisenstein series must have  a pole at
$s=s_0$.

Second assume that $\dim L=1$. In Sec.~\ref{sec:dim-L=1} we further cut $I_{\Omega_{0,1}}(\xi)$ and $I_{\Omega_{1,1}}(\xi)$ into 3 parts each and computed the values. Using the notation there we have
\begin{equation*}
  I_{\Omega_{0,1}}(\xi) = J_{\Omega_{0,1},1}(\xi) +J_{\Omega_{0,1},2}(\xi) + J_{\Omega_{0,1},3}(\xi). 
\end{equation*}
Proposition~\ref{prop:van-I-Omega-10-L=1} shows that $I_{\Omega_{1,0}}(\xi)$ vanishes, while Propositions~\ref{prop:van-J-Omega-11-1}, \ref{prop:van-J-Omega-11-2}, \ref{prop:van-J-Omega-11-3} show that all 3 parts of $I_{\Omega_{1,1}}(\xi)$  vanish. Propositions~\ref{prop:van-J-Omega-01-1}, \ref{prop:van-J-Omega-01-3} show that $J_{\Omega_{0,1},1}(\xi)$ and $J_{\Omega_{0,1},3}(\xi)$ vanish. As we assume that  $M (w,s)$ does not have a pole at
$s=s_0$, we see from \eqref{eq:J-Omega01-2-computed-expression} that $J_{\Omega_{0,1},2}(\xi_s^c)$ does not contribute a pole at $s=s_0$.
Finally Proposition~\ref{prop:non-van-of-key-period-L=1} shows that $J_{\Omega_{0,1},2}(\xi_{c,s})$ has a pole at $s=s_0$ for some choice of data. Thus the truncated Eisenstein series and also the Eisenstein series must have  a pole at
$s=s_0$. 

In both cases, we find that the Eisenstein series has a pole at $s=s_0$, so this contradicts our assumption that $M (w,s)$ doe not have a pole at $s=s_0$. Thus the intertwining operator
$M(w,s)$ must have a pole at $s=s_0$. Therefore the Eisenstein series  must
have a pole there. We note that the conditions of the propositions we have invoked are satisfied by the conditions on the distinction and the non-distinction.
\end{proof}
For a complex number $s_1$, let $\cE_{s_1} (g,f_s)$ denote the residue of $E^{Q_1} 
(g,s,f_s)$ at $s=s_1$.  As we will see in the proof of the following theorem, the value of the period integral 
\begin{equation}\label{eq:period-residue-2}
  \int_{[J (Z_1,L)]} \cE_{s_0}(g,f_s)\theta_{\psi,X_1,Y} (g,1,\Phi) dg,
\end{equation}
which is regularised via Arthur truncation does not depend on the truncation parameter. Thus we may  extend the notion of distinction. We say that $\cE_{s_0}(g,f_s)\theta_{\psi,X_1,Y} (g,1,\Phi)$ is $[J (Z_1,L)]$-distinguished if \eqref{eq:period-residue-2} is non-vanishing.

\begin{thm} \label{thm:dist->period-residue-Eis}
With the same setup as in Theorem~\ref{thm:dist->Eis-pole}, assume further that $s_0 >0$. Then the following hold.
  \begin{enumerate}
  \item $s_0$ is in $\cP_1 (\chi,\sigma)$ and it corresponds to a
simple pole.
  \item $\cE_{s_0}(g,f_s)\theta_{\psi,X_1,Y} (g,1,\Phi)$ is $J (Z_1,L)$-distinguished for some choice of
    $f_s\in\cA_1 (s,\chi_{Y},\sigma)$ and $\Phi\in \cS_{X_1,Y} (\A)$.
    \item For any choice of  $f_s\in\cA_1 (s,\chi_{Y},\sigma)$ and $\Phi\in \cS_{X_1,Y}(\A)$, the function
    $\cE_{s_1}(g,f_s)\theta_{\psi,X,Y} (g,1,\Phi)$
     can not be $J (Z_1, L)$-distinguished at any $0<s_1\neq s_0$.
\item For any $(Z'',L'')$ with $\dim L''=0$ or $1$ and $\dim Z'' - \dim L'' > \dim Z -\dim L$ and for any choice of  $f_s\in\cA_1 (s,\chi_{Y},\sigma)$ and $\Phi\in \cS_{X_1,Y}(\A)$,
$\cE_{s_1}(g,f_s)\theta_{\psi,X_1,Y} (g,1,\Phi)$ can not be $J (Z''_1,L'')$-distinguished at any $s_1>0$.
\end{enumerate}
\end{thm}
\begin {proof}    The computation of the theta-twisted periods of truncated Eisenstein series is done in Sec.~\ref{sec:value}. The periods are sums of  several  integrals that are absolutely convergent for $\Re s$ large enough. In addition most of them vanish. (See proof of Thm.~\ref{thm:dist->Eis-pole}.) We essentially just read off the residues of those non-vanishing integrals to show this theorem.

  We make use of the notation introduced in
   Sec.~\ref{sec:arthur-trunc}. Let $s_1$ be a positive real number.
   Set
   \begin{align*}
     \theta^c (g) = &\sum_{\gamma \in Q_1 \lmod G (X_1)} \cE_{s_1,Q_1} (\gamma g,f_s) \hat
{\tau}^c (H (\gamma g))\\ = &\sum_{\gamma \in Q_1
\lmod G (X_1)} \res_{s=s_1} (M (w,s)f_s (\gamma g)) \hat {\tau}^c (H
(\gamma g)).
   \end{align*}
   Then the truncated residue is given by
   \begin{equation*}
     \Lambda^c \cE_{s_1} (g,f_s) = \cE_{s_1} (g,f_s) - \theta^c (g).
   \end{equation*}
The period
\begin{equation*} \int_{[J (Z_1,L)]} \theta^c (g) \theta_{\psi,X_1,Y}
(g, 1, \Phi) dg
\end{equation*} is absolutely convergent when $\Re s$ and the
truncation parameter $c$ are sufficiently large, has meromorphic continuation to the complex plane and as $s_1>0$, is equal to
$\res_{s=s_1} I (\xi^c_s)$ by Sections~\ref{sec:abs-conv-issue} and \ref{sec:value}.
Thus we find that
\begin{align}\label{eq:period-residue} &\int_{[J (Z_1,L)]} \cE_{s_1} (g,f_s)
                                         \theta_{\psi,X_1,Y} (g, 1, \Phi) dg\\
  =& \int_{[J (Z_1,L)]} (\Lambda^c
\cE_{s_1} (g,f_s) + \theta^c (g)) \theta_{\psi,X_1,Y_0} (g, 1, \Phi)
     dg\nonumber\\
  =& \res_{s=s_1} (I (\xi_{c,s})- I (\xi_s^c)) + \res_{s=s_1} I (\xi_s^c)\nonumber\\
=& \res_{s=s_1} I (\xi_{c,s})\nonumber.
\end{align}
When $s_1 \neq s_0$, the
period \eqref{eq:period-residue} vanishes. This proves part (3).
When $s_1=s_0$, the period \eqref{eq:period-residue} is non-vanishing for some choice of data by Prop~\ref{prop:non-van-of-key-period-L=0} for $\dim L=0$ and by Prop.~\ref{prop:non-van-of-key-period-L=1} for $\dim L =1$. These two propositions also show that the values are independent of the truncation parameter. This proves part (2) and as a result, part (1). For part (4), as the domain of integration is over a `larger' group, Section~\ref{sec:value} coupled with the assumptions on non-distinction shows that the period analogous to \eqref{eq:period-residue} vanishes.
\end{proof}

\begin{cor}\label{thm:FO-Eis-pole} Let $\sigma\in\cA_\cusp (G (X))$ and $Y$ a (possibly trivial)
anisotropic quadratic space. Assume that $\FO_{\psi}^{Y} = \dim Y +
2r$. Let $s_0=\half (\dim X - (\dim Y +2r) +2)$.
Then $s_0$ is a  pole of $E^{Q_1}
(s,g,f_s)$ for some $f_s\in \cA_1 (s,\chi_{Y},\sigma)$.
Assume further that
$s_0 >0$. Then the following hold.
  \begin{enumerate}
  \item $s_0$ is in $\cP_1 (\chi,\sigma)$ and it corresponds to a
simple pole.
\item For some choice of $(Z, L)$ with $\dim L=0$ or $1$ and $\dim X - \dim Z + \dim L = r$, $f_s\in\cA_1 (\chi_{Y},\sigma)$
and $\Phi\in \cS_{X_1,Y} (\A)$, $\cE_{s_0}
(g,f_s)\theta_{\psi,X_1,Y} (g,1,\Phi)$
     is $J (Z_1,L)$-distinguished.
  \item For $0<s_1\neq s_0$, the function $\cE_{s_1}
(g,f_s)\theta_{\psi,X_1,Y} (g,1,\Phi)$ can not be
     $J (Z_1, L)$-distinguished for any choice
     of $(Z,L)$ with $\dim L=0$ or $1$ and $\dim X - \dim Z + \dim L = r$, $f_s\in\cA_1 (\chi_{Y},\sigma)$ and $\Phi\in \cS_{X_1,Y}
(\A)$.
\item For any $s_1>0$, $\cE_{s_1}(g,f_s)\theta_{\psi,X_1,Y} (g,1,\Phi)$ cannot be $J (Z''_1,L'')$-distinguished for any $(Z'',L'')$ with $\dim L''=0$ or $1$ and $\dim X - \dim Z'' + \dim L'' < r$, $f_s\in\cA_1 (\chi_{Y},\sigma)$ and $\Phi\in \cS_{X_1,Y}
(\A)$.
  \end{enumerate}
\end{cor}
\begin{proof}
  By Prop.~\ref{prop:FO->period-from-fourier-coeff}, the assumptions imply that $\sigma\otimes\Theta_{\psi,X,Y}$ is $J (Z,L)$-distinguished for some $(Z,L)$ such that $\dim L =0$ or $1$ and $\dim X -\dim Z + \dim L = r$ and  that it is not $J (Z',L')$-distinguished for all $(Z',L')$  with $\dim X - \dim Z' + \dim L' < r$. Hence the assumptions of Thm.~\ref{thm:dist->Eis-pole} or Thm.~\ref{thm:dist->period-residue-Eis} are satisfied. An application of the theorems gives our results.
\end{proof}

Now we can strengthen our results in Thm.~\ref{thm:Eis-pole-LO}.
\begin{thm}\label{thm:Eis-pole-LO-strengthened}
  Let $\sigma\in \cA_\cusp (G (X))$
and $s_0$ be the maximal element in $\cP_1 (\chi,\sigma)$. Write $s_0 = \half (\dim X +2)-j$. Then
\begin{enumerate}
\item $j$ is an integer such that $\frac{1}{4}\dim X \le j<\half (\dim X +2)$;
\item $\LO_{\psi,\chi} (\sigma) = 2j$.
\end{enumerate}
\end{thm}
\begin{proof}
  Part (1) is a restatement of Thm.~\ref{thm:Eis-pole-LO}. We show part (2). We have already shown that $\LO_{\psi,\chi} (\sigma) \le 2j$. Assume that the lowest occurrence  is achieved in the Witt tower of $Y$ so that  we have
    $\LO_{\psi,\chi} (\sigma)=\FO_\psi^Y (\sigma) = 2j' \le 2j$.
  However, by Cor.~\ref{thm:FO-Eis-pole}, for some choice of $f_s\in \cA_1 (s,\chi_{Y},\sigma)$, $E^{Q_1}(s,g,f_s)$ has a pole at
  $s=\half (\dim X + 2 - 2j')$. Thus $j\le j'$ by maximality of $s_0$. Hence we are forced to have $j'=j$. In other words, $\LO_{\psi,\chi} (\sigma) =2j$.
\end{proof}



\section{Computation of Values of Integrals }
\label{sec:value}
Assuming absolute convergence of relevant integrals, we proceed to compute values of the integrals
\eqref{eq:I-Omega01}, \eqref {eq:I-Omega10} and \eqref {eq:I-Omega11}. Thus we will exchange orders of integrations and/or summations freely. First we set up some notation.
For $t\in \GL_1$, let $m_1
(t)$ denote the element in $G (Fe_1^+\oplus Fe_1^-)$ given by
$  \smatrix{t}{0}{0}{t^{-1}}$.
We also write $m_1 (t)$ for its natural image in $G (X_1)$.
For a symplectic space $V$ and a totally isotropic subspace $L$ in $V$, let $Q (V,L)$ be the parabolic subgroup of $G (V)$ that stabilises $L$,  $M (V,L)$ be the Levi part, $N (V,L)$ be the unipotent radical and $\overline{N}(V,L)$ the opposite unipotent subgroup.
In this section, $Z$ (resp. $Z'$) is
always a non-degenerate subspace of $X$ and $L$ (resp. $L'$) is
a totally isotropic subspace in $Z$ (resp. $Z'$). Set $r = \dim X - \dim Z +\dim L$ and $s_0=\half (\dim X - (\dim Y +2r) +2)$. We separate the computation into two cases.

\subsection{Case $\dim L = 0$}
\label{sec:dim-L=0}

\subsubsection{$I_{\Omega_{0,1}}$} We need to consider the integral
\eqref{eq:I-Omega01}. We compute
\begin{align} I_{\Omega_{0,1}} (\xi) = &\int_{[G (Z_1)]}
\sum_{\delta\in Q (Z_1,Fe_1^-) \lmod G (Z_1)} \xi (\delta g) dg
\nonumber\\ = &\int_{Q (Z_1,Fe_1^-)\lmod G (Z_1) (\A)} \xi (g) dg
\nonumber\\ = &\int_{K_{G (Z_1)}}\int_{[Q (Z_1,Fe_1^-)]} \xi (qk) dq
dk \nonumber\\
\label{eq:I-Omega01-for-abs-conv} = &\int_{K_{G (Z_1)}} \int_{[G (Z)]}
\int_{[\GL_1]} \int_{[N(Z_1,Fe_1^-)]} \xi( n m_1 (t)h
k)|t|_\A^{-2\rho_{Q (Z_1,Fe_1^-)}} dn dt dh dk.
\end{align} Let $\xi = \xi_{c, s}$. Then
\begin{align*} \xi (n m_1 (t)h k)= &f_s (n m_1 (t)h k)\theta_{X_1,Y}
(n m_1 (t)h k, 1, \Phi)\hat {\tau}_c (H (n m_1 (t)h k))\\ =& \chi_Y
(t) |t|_\A^{s+ \rho_{Q_1}} f_s (hk)\theta_{X_1,Y} (n m_1 (t)h k, 1, \Phi) \hat
{\tau}_c (H ( m_1 (t))).
\end{align*} Only $\theta_{X_1,Y}$ involves $n$. We realise the Weil representation on $\cS ((Y\otimes\ell_1^+\oplus (Y\otimes X)^+) (\A))$. Using the explicit
formula for Weil representation and the fact that $Y$ is anisotropic,
the inner integral over $[N(Z_1,Fe_1^-)]$ is
\begin{align*} & \int_{[N(Z_1,Fe_1^-)]}\theta_{X_1,Y} (n m_1
(t)h k, 1, \Phi) dn \\ =&\int_{[\Hom (\ell_1^+, Z)]} \sum_{w\in
(Y\otimes X)^+}\omega_{X_1,Y} (n (\mu,0) m_1 (t)h k,1)\Phi (0,w) d\mu
\\ =& \sum_{w\in (Y\otimes X)^+}\omega_{X_1,Y} ( m_1 (t)h k,1)\Phi
(0,w)\\ =& \chi_{Y} (t)|t|_{\A}^{\dim Y/2}\theta_{X,Y} ( h, 1, \Phi_k)
\end{align*} where $\Phi_k (\cdot) = \omega_{X_1,Y} ( k,1)\Phi
(0,\cdot)$ and for $\mu\in \Hom (\ell_1^+, Z)$, $n (\mu,0)$ denotes the element in $N (Z_1,Fe_1^-)$ that sends $e_1^+$ to $e_1^+ + \mu(e_1^+)$ and that acts on $Z$ by duality. Then we integrate over $[\GL_1]$. We just focus on the terms
involving $t$:
\begin{align*} & \int_{[\GL_1]} \chi_{Y} (t)
|t|_{\A}^{s+\rho_{Q_1}}\chi_{Y} (t)|t|_{\A}^{\dim Y/2} \hat {\tau}_c
                 (H (m_1 (t)))|t|_\A^{-2\rho_{Q (Z_1,Fe_1^-)}} d t\\
 =& \vol ( F^\times \lmod \A^1) \int_{0}^c
t^{s-s_0} d^\times t.
\end{align*} Thus for $\Re s > s_0 $ we
have
\begin{align*} I_{\Omega_{0,1}} (\xi_{c,s}) =  \vol ( F^\times \lmod
\A^1) \frac{c^{s-s_0}}{s-s_0}\int_{K_{G (Z_1)}}\int_{[G (Z)]} f_s
(hk)\theta_{X,Y} ( h, 1, \Phi_k) dh dk.
\end{align*} This expression  provides a meromorphic continuation of $I_{\Omega_{0,1}} (\xi_{c,s})$ as function in $s$. An analogous computation shows that for $\Re s > -s_0 $,
\begin{align}\label{eq:I-Omega01-computed-expression} I_{\Omega_{0,1}} (\xi^c_s) =  \vol ( F^\times \lmod
\A^1) \frac{c^{-s-s_0}}{s+s_0} \int_{K_{G (Z_1)}}\int_{[G (Z)]} M (w,s)
f_s (hk)\theta_{X,Y} ( h, 1, \Phi_k) dh dk.
\end{align}
Note that the inner  integrals are period integrals on $\sigma\otimes \Theta_{\psi,X,Y}$.

\begin{prop} \label{prop:non-van-of-key-period-L=0}
  \begin{enumerate}
  \item If $\sigma\otimes \Theta_{\psi,X,Y}$ is not $G (Z)$-distinguished, then both $I_{\Omega_{0,1}} (\xi_{c,s})$ and $I_{\Omega_{0,1}} (\xi^c_s)$ vanish identically.
    \item $I_{\Omega_{0,1}} (\xi_{c,s})$ does not have a pole at  $s\neq s_0$.
    \item If $\sigma\otimes \Theta_{\psi,X,Y}$ is $G (Z)$-distinguished, then the residue
      \begin{equation*}
        \vol ( F^\times \lmod \A^1)\int_{K_{G (Z_1)}}\int_{[G (Z)]} f_{s_0}
(hk)\theta_{X,Y} ( h, 1, \Phi_k) dh dk
\end{equation*}
of $I_{\Omega_{0,1}} (\xi_{c,s})$ at $s=s_0$ does not vanish for some choice of data. Here
$\Phi_k (\cdot) = \omega_{X_1,Y} ( k,1)\Phi(0,\cdot)$.
  \end{enumerate}
\end{prop}
\begin{proof}
  We just need to show  part (3). We will suppress $\psi$ to save space. Since we assume that $\sigma\otimes \Theta_{X,Y}$ is $G (Z)$-distinguished, there exist $\phi\in\sigma$ and $\Psi\in\cS_{X,Y} (\A)$ such that
  \begin{equation*}
    \int_{[G (Z)]} \phi(h)\theta_{X,Y} ( h, 1, \Psi) dh
  \end{equation*}
  is non-vanishing.
  We may assume that $\Psi$ is a pure tensor of the form $\otimes_v \Psi_v$. We will construct $f_s\in\cA_1 (s,\chi,\sigma)$ from $\phi$ and $\Phi\in\cS_{X_1,Y} (\A)$ from $\Psi$. Let $S_\ur$ be the set of finite places of $F$ such that $\chi_v$ and $\psi_v$ are unramified, $\phi$ is invariant under $K_{G (X),v}$ and $\Psi_v$ is the characteristic function of $Y\otimes X^+ (\cO_v)$. Let $S_\rf$ be the rest of the  finite places of $F$. Let $S_\infty$ be the set of infinite places of $F$. Set $S_\ram = S_\rf \cup S_\infty$. We follow the idea in  \cite [Prop.~2] {MR1142486}, but we need finer analysis to account for the theta twist.

  Assume that  $v\in S_\rf$. Choose a small compact open subgroup $\Omega'_v$ of $\overline {N}_{1,v}$ such that $\Omega_v := (Q_{1,v}\cap K_{G (X_1),v}) \cdot \Omega'_v \subset K_{G (X_1),v}$. For $k_v\in \Omega_v$ we write
    $$k_v = n_{k_v}m_1 (a_{k_v})h_{k_v} \omega'_{k_v}$$ 
with $n_{k_v}\in N_{1,v}\cap K_{G (X_1),v}$, $a_{k_v} \in \cO_v^\times$, $h_{k_v}\in G (X)_v \cap K_{G (X_1),v}$ and $\omega'_{k_v} \in \Omega'_v$. 
Assume  $v\in S_\infty$. Fix a complement of $\Lie (M_{1,v}\cap K_{G (X_1),v})$ in $\Lie (K_{G (X_1),v})$, choose a small open subset in this complement and let $\Omega'_v$ denote its exponential. Set $\Omega_v = (M_{1,v}\cap K_{G (X_1),v}) \cdot \Omega'_v$. Also choose a smooth function $\alpha_v$ supported in $\Omega'_v$ with $\alpha_v (I)=1$. For $k_v\in \Omega_v$ we write
$k_v = m_1 (a_{k_v})h_{k_v} \omega'_{k_v}$
with $a_{k_v} \in \cO_v^\times$, $h_{k_v}\in G (X)_v \cap K_{G (X_1),v}$ and $\omega'_{k_v} \in \Omega'_v$. Note that $f_s$ is determined by its values on $K_{G (X_1)}$. For  $k \in K_{G (X_1)}$, we  determine $f_s (k)$ as follows
  \begin{align*}
    &f_s (k) =  f_s (\prod_{v\in S_\ram}k_v) \\
    =&
       \begin{cases}\displaystyle
         \phi (\prod_{v\in S_\ram} h_{k_v})\chi|\cdot|_\A^{s+\rho_{Q_1}} (\prod_{v\in S_\ram} a_{k_v}) \prod_{v\in S_\infty}\alpha_v (\omega'_{k_v}) \quad &\text{if $k_v\in \Omega_v, \forall v\in S_\ram$;
                                                                                                                                                                       }\\
         0 \quad &\text{otherwise.}
       \end{cases}
  \end{align*}

  Now construct a Schwartz function $\Psi_0 = \otimes_v \Psi_{0,v} \in \cS (Y\otimes Fe_1^+ (\A))$ and set $\Phi = \Psi_0 \otimes \Psi$. Assume $v\in S_\ur \cup S_\rf$. Take $\Psi_{0,v}\in \cS (Y\otimes Fe_1^+ (F_v))$ to be the characteristic function of $Y\otimes Fe_1^+ (\cO_v)$. For $v\in S_\rf$, we may shrink $\Omega'_v$ such that $\Phi_{0,v}$ is invariant under the action of $\Omega'_v$. For $v\in S_\infty$, take $\Psi_{0,v} \in \cS (Y\otimes Fe_1^+ (F_v))$ to be any Schwartz function taking value $1$ at the origin. Write $\Omega'_\infty = \prod_{v\in S_\infty} \Omega'_v$ and $\alpha_\infty = \prod_{v\in S_\infty} \alpha_v$ etc. It can be checked that the residue in question is equal to a non-zero constant times
  \begin{align*}
    &\int_{K_{G (Z_1)}}\int_{[G (Z)]} f_s(hk)\theta_{X,Y} ( h, 1, \Phi_k) dh dk \\
    =&\int_{\Omega'_\infty}\int_{[G (Z)]} \phi(h)
      \theta_{X,Y} ( h, 1, \Phi_{\omega'_\infty})\alpha_\infty (\omega'_\infty) dh d\omega'_\infty.
  \end{align*}
  Thus if we pick $\Omega'_\infty$ to be small enough, then, by continuity, the above is non-vanishing.
\end{proof}

\subsubsection{$I_{\Omega_{1,0}}$}
\label{sec:I-Omega10}

We need to consider the integral \eqref{eq:I-Omega10}. For each $v\in V-\{0\}$, fix a dual vector $v^+
\in V$ such that $\form {v^+} {v}_V = 1$. Then we may choose
$\gamma_{Fv}$ to be the element in $G (X_1)$ that is determined by the
following action:
$e_1^+ \leftrightarrow v^+$ and $ e_1^- \leftrightarrow v$, 
with identity action on orthogonal complement. Then
$\gamma_{Fv} G (Z_1) \gamma_{Fv}^{-1} = G (Fv^+ \oplus Z \oplus
Fv)$. Thus
\begin{equation*} I_{\Omega_{0,1}} (\xi) = \sum_{v \in V - \{0\} /
F^\times}\int_{[G (Fv^+ \oplus Z \oplus Fv)]} \xi (g\gamma_{Fv})dg
\end{equation*}
and thus we get
\begin{prop}\label{prop:van-I-Omega-10}
  If $\sigma\otimes \Theta_{\psi,X,Y}$ is not $G (Z')$-distinguished for all $Z' \supset Z$ such that $\dim Z' = \dim Z +2$, then $I_{\Omega_{1,0}} (\xi)$ vanishes for $\xi = \xi_{c,s}$ and $\xi_s^c$.
\end{prop}

\subsubsection{$I_{\Omega_{1,1}}$}
\label{sec:I-Omega11}
We need to consider \eqref{eq:I-Omega11}. For
each $v\in V-\{0\}$, fix a dual vector $v^+ \in V$ such that $\form {v^+} {v}_V =
1$. We take $\gamma_{F (v+e_1^-)}$ to be the element in $G (X_1)$ that
is determined by the following action:
$ e_1^+ \mapsto v^+$, $v^+ \mapsto e_1^+ - v^+$, $ v\mapsto e_1^-$, and $ e_1^- \mapsto  v+e_1^-$, with identity action on orthogonal complement. With
respect to the basis $e_1^+, v^+, v, e_1^-$, it can be written as
$\gamma_1\gamma_2$ where
\begin{equation}\label{eq:gamma_F-v+e1-=gamma1gamma2} \gamma_1 =
\left(
  \begin{array}{cc|cc} &1&&\\ 1&&&\\ \hline &&&1\\ &&1&
  \end{array} \right), \qquad \gamma_2= \left( \begin{array}{cc|cc}
1&-1&&\\ &1&&\\ \hline &&1&1\\ &&&1
   \end{array} \right).
\end{equation} Using Iwasawa decomposition we find that each $v$-term
in \eqref{eq:I-Omega11} is equal to
\begin{equation*} \int_{K_{G (Z_1)}} \int_{\GL (\A)} \int_{[J (Z_1,
Fe_1^-)]} \xi (\gamma_{F (v+e_1^-)}g m_1 (t) k) |t|_\A^{-2\rho_{Q
(Z_1,Fe_1^-)}} dg d^\times t dk.
\end{equation*} We can check that $\gamma_2$ commutes with each
element in $J (Z_1, Fe_1^-)$ and that $\gamma_1 J (Z_1, Fe_1^-)
\gamma_1^{-1} = J (Fv^+ \oplus Z \oplus Fv, Fv)$. Thus we get the inner
integral
\begin{equation*} \int_{[J (Fv^+ \oplus Z \oplus Fv, Fv)]} \xi (g
\gamma_{F (v+e_1^-)} m_1 (t) k)dg.
\end{equation*}
Therefore we have:
\begin{prop}\label{prop:van-I-Omega-11}
  If $\sigma\otimes \Theta_{\psi,X,Y}$ is not $J (Z', L')$-distinguished for all $Z' \supset Z$ such that $\dim Z' = \dim Z +2$ and $L'$ an isotropic line of $Z'$ in the orthogonal complement of $Z$, then $I_{\Omega_{1,1}} (\xi)$ vanishes for $\xi = \xi_{c,s}$ and $\xi_s^c$.
\end{prop}

\subsection{Case $\dim L =1$}
\label{sec:dim-L=1}

Take a non-zero vector $f_1^- \in L$. Fix $f_1^+\in Z$ dual to $f_1^-$
with $\form {f_1^+} {f_1^-}_Z=1$. Write $Z = Ff_1^+ \oplus W \oplus
Ff_1^-$. Set $W_1 = Fe_1^+ \oplus W \oplus Fe_1^-$.

\subsubsection{$I_{\Omega_{01}}$} We need to further analyse the integral
\eqref{eq:I-Omega01}.
The generalised flag variety $Q (Z_1,Fe_1^-) \lmod G
(Z_1)$ parametrises isotropic lines in $Z_1$. For an isotropic line
$\ell\in Z_1$ we set $\delta_\ell$ to be an element in $G (Z_1)$ such
that $e_1^-\delta_\ell = \ell$. We consider the $J
(Z_1,L)$-orbits of isotropic lines. An isotropic line $F (af_1^+ + w + bf_1^-)$ for
$a,b\in F$ and $w\in W_1$ is in the same $J (Z_1,L)$-orbit as
\begin{align*} Ff_1^+, \quad &\text{if $a\neq 0$};\\ Fe_1^-, \quad
&\text{if $a= 0$ and $w\neq 0$};\\ Ff_1^-, \quad &\text{if $a= 0$ and
$w= 0$}.
\end{align*} The stabiliser of $Ff_1^+$ in $J (Z_1,L)$ is $G (W_1)$. This implies that
the quotient $\Stab_{J (Z_1,L)} (Ff_1^+) \lmod J (Z_1,L) \isom N
(Z_1,L)$. Set $NJ_1$ to be the subgroup of $J (Z_1,L)$ consisting of
elements of the form
\begin{equation*}
  \begin{pmatrix} 1 & * &0 &0&0\\ & 1 &0&0&0\\ &&I &0&0\\ &&&1& *\\
&&&&1
  \end{pmatrix}
\end{equation*} with respect to the `basis' $f_1^+, e_1^+, W, e_1^-,
f_1^-$ and $NJ_2$ to be the subgroup of $J (Z_1,L)$ consisting of
elements of the form
\begin{equation*}
  \begin{pmatrix} 1 & 0 &* &*&*\\ & 1 &0&0&*\\ &&I &0&*\\ &&&1&0 \\
&&&&1
  \end{pmatrix}.
\end{equation*} Then $\Stab_{J (Z_1,L)} (Fe_1^-) \isom NJ_2 \rtimes Q
(W_1,Fe_1^-)$. For $Ff_1^-$ the stabiliser is $J (Z_1,L)$. Thus
\eqref{eq:I-Omega01} further splits into three parts:
$I_{\Omega_{0,1}} = J_{\Omega_{0,1},1} + J_{\Omega_{0,1},2} +
J_{\Omega_{0,1},3} $ where
\begin{align}
\label{eq:J-Omega01-1} J_{\Omega_{0,1},1} (\xi) =&\int_{[J (Z_1,L)]}
\sum_{\eta \in N (Z_1, L)} \xi (\delta_{Ff_1^+}\eta g) dg;\\
\label{eq:J-Omega01-2} J_{\Omega_{0,1},2} (\xi) =&\int_{[J (Z_1,L)]}
\sum_{\eta \in NJ_2 \rtimes Q (W_1,Fe_1^-) \lmod J (Z_1,L)} \xi (\eta
g) dg;\\
\label{eq:J-Omega01-3} J_{\Omega_{0,1},3} (\xi) =&\int_{[J (Z_1,L)]}
\xi (\delta_{Ff_1^-} g) dg.
\end{align}
Now we compute each of the three integrals.

The integral \eqref{eq:J-Omega01-1} is equal to:
$\int_{N (Z_1,L)(\A)} \int_{[G (W_1)]} \xi
  (\delta_{Ff_1^+} g n ) dg dn$.
We pick $\delta_{Ff_1^+}$ to be the element in $G (Z_1)$
that is determined by the following action:
$ e_1^+ \leftrightarrow -f_1^-$ and $ e_1^- \leftrightarrow f_1^+$, 
with identity action on orthogonal complement. Then
$\delta_{Ff_1^+} G (W_1) \delta_{Ff_1^+}^{-1} = G (Ff_1^+ \oplus W
\oplus Ff_1^- ) =G (Z)$. Thus we get an inner integral
$\int_{[G (Z)]} \xi (g \delta_{Ff_1^+} n ) dg$, and we have:
\begin{prop}\label{prop:van-J-Omega-01-1}
  If $\sigma\otimes \Theta_{\psi,X,Y}$ is not $G (Z)$-distinguished, then $J_{\Omega_{0,1},1} (\xi)$ vanishes for $\xi = \xi_{c,s}$ and $\xi_s^c$.
\end{prop}

The integral \eqref{eq:J-Omega01-2} is equal to:
\begin{align*} & \int_{NJ_1 (\A)} \int_{[NJ_2]} \int_{Q (W_1,Fe_1^-)
                 \lmod G (W_1) (\A)} \xi (n_2 n_1 g) dg dn_2 dn_1 \\
  =&\int_{K_{G(W_1)}} \int_{NJ_1 (\A)} \int_{[NJ_2]} \int_{[N (W_1,Fe_1^-)]}
\int_{[G (W)]} \int_{[\GL_1]}\xi (n_2 n_1 n m_1 (t) g k)\\
               &\qquad \qquad\qquad |t|_\A^{-2\rho_{Q (W_1,Fe_1^-)}} dt dg dn dn_2 dn_1 dk\\
 =&\int_{K_{G (W_1)}} \int_{NJ_1 (\A)} \int_{[N(Z_1,Fe_1^-\oplus
   Ff_1^-)]} \int_{[G (W)]}\int_{[\GL_1]} \xi (n_3n_1 m_1 (t) g k)\\
  &\qquad \qquad\qquad |t|_\A^{-2\rho_{Q (W_1,Fe_1^-)}} dt dg dn_3 dn_1 dk.
\end{align*}
We conjugate $m_1 (t)$
across $n_1$, which produces a modular character, and  exchange
$n_1$ and $g$ since they commute. We get
\begin{align*} &\int_{K_{G (W_1)}} \int_{NJ_1 (\A)}
\int_{[N(Z_1,Fe_1^-\oplus Ff_1^-)]} \int_{[G (W)]} \int_{[\GL_1]}\xi (
                 n_3 m_1 (t) g n_1 k) \\
  &\qquad \qquad\qquad|t|_\A^{-2\rho_{Q (W_1,Fe_1^-)}-1} dt dg
                 dn_3 dn_1 dk\\
  =&\int_{K_{G (W_1)}} \int_{NJ_1 (\A)} \int_{[N' ]}
\int_{[N(Z, Ff_1^-)]} \int_{[G (W)]} \int_{[\GL_1]} \xi ( n' m_1 (t) v
     g n_1 k) \\
  &\qquad \qquad\qquad |t|_\A^{-2\rho_{Q (W_1,Fe_1^-)} -1} dt dg dv dn' dn_1
     dk\\
  =&\int_{K_{G (W_1)}} \int_{NJ_1 (\A)} \int_{[N' ]} \int_{[J(Z,
     Ff_1^-)]} \int_{[\GL_1]} \xi ( n' m_1 (t) g n_1 k) \\
  &\qquad \qquad\qquad |t|_\A^{-2\rho_{Q
(W_1,Fe_1^-)} -1} dt dg dn' dn_1 dk
\end{align*} where we set $N' = N (Z_1, Fe_1^-) \cap N
(Z_1,Fe_1^-\oplus Ff_1^-)$.

Suppose that $\xi = \xi_{c,s}$. Then
\begin{align*}& \xi_{c,s} (n' m_1 (t) g n_1 k) \\= &f_s (n' m_1 (t) g n_1
k)\theta_{X_1,Y} (n' m_1 (t) g n_1 k, 1, \Phi)\hat {\tau}_c (H (n' m_1
(t) g n_1 k)) \\ = &\chi_Y (t)|t|_\A^{s+\rho_{Q_1}} f_s ( g
n_1 k)\theta_{X_1,Y} (n' m_1 (t) g n_1 k, 1, \Phi)\hat {\tau}_c (H (
m_1 (t) g n_1)).
\end{align*} Only the term $\theta_{X_1,Y}$ involves $n'$. By explicit
formula of Weil representation, we get
\begin{equation*} \int_{[N' ]} \theta_{X_1,Y} (n' m_1 (t) g n_1 k, 1,
\Phi) dn' = \chi_Y (t) |t|_\A^{\dim Y /2}\theta_{X,Y} (g,1,\Phi_{n_1
k})
\end{equation*} where $\Phi_{n_1 k} = \omega_{X_1,Y} ( n_1 k, 1)\Phi
(0,\cdot)$. In total, the exponent of $|t|_\A$ is
\begin{equation*}
  s+\rho_{Q_1}+\dim Y/2 -2\rho_{Q (W_1,Fe_1^-)} -1 = s - s_0.
\end{equation*}
 It can be
checked that $H ( m_1 (t) g n_1) = H ( m_1 (t) n_1)$.  Set $c (n_1) =\exp (\form {H (n_1)} {1})$. Write $\overline {n}$ for $n_1$ to emphasise the fact that $NJ_1$ is a subgroup of the opposite unipotent. We find that
$J_{\Omega_{0,1},2} (\xi_{c,s})$ is equal to
\begin{align*} \int_{K_{G (W_1)}} \int_{NJ_1 (\A)} \int_{[J(Z,L)]} f_s ( g \overline {n} k) \theta_{X,Y} (g,1,\Phi_{\overline {n} k})
 \int_{[\GL_1]}\hat {\tau}_c (H ( m_1 (t) \overline {n})) |t|_\A^{s - s_0} dt dg d\overline {n} dk\\
  = \vol ( F^\times \lmod \A^1)\int_{K_{G (W_1)}} \int_{NJ_1 (\A)} \int_{[J(Z,
    L)]} f_s ( g \overline {n} k) \theta_{X,Y} (g,1,\Phi_{\overline {n} k})
    \frac{(c/c (\overline {n}))^{s-s_0}}{s-s_0} dg d\overline {n} dk.
\end{align*}
 Thus we find that the above can possibly have a pole only at $s=s_0$. The residue is
\begin{align*}
\vol ( F^\times \lmod \A^1)  \int_{K_{G (W_1)}} \int_{NJ_1 (\A)} \int_{[J(Z,L)]} f_{s_0} ( g \overline {n} k) \theta_{X,Y} (g,1,\Phi_{\overline {n} k})
  dg d\overline {n} dk.
\end{align*}
 Note that the inner most integral is a period integral on $\sigma\otimes \Theta_{\psi,X,Y}$.

Similarly we can compute $J_{\Omega_{0,1},2} (\xi_s^c)$:
\begin{multline}\label{eq:J-Omega01-2-computed-expression}
 \vol ( F^\times \lmod \A^1) \int_{K_{G (W_1)}} \int_{NJ_1 (\A)} \int_{[J(Z,
    L)]} M (w,s) f_s ( g \overline {n} k) \theta_{X,Y} (g,1,\Phi_{\overline {n} k})
    \\ \frac{(c/c (\overline {n}))^{-s-s_0}}{s+s_0} dg d\overline {n} dk.
\end{multline}
\begin{prop}\label{prop:non-van-of-key-period-L=1}
  \begin{enumerate}
  \item
    If $\sigma\otimes \Theta_{\psi,X,Y}$ is not $J (Z,L)$-distinguished, then both $J_{\Omega_{0,1},2} (\xi_{c,s})$ and $J_{\Omega_{0,1},2} (\xi_s^c)$ vanish identically;
    \item $J_{\Omega_{0,1},2} (\xi_{c,s})$ does not have a pole at  $s\neq s_0$;
    \item If $\sigma\otimes \Theta_{\psi,X,Y}$ is  $J (Z,L)$-distinguished, then  the residue
      \begin{equation*}
        \vol ( F^\times \lmod \A^1)  \int_{K_{G (W_1)}} \int_{NJ_1 (\A)} \int_{[J(Z,L)]} f_{s_0} ( g \overline {n} k) \theta_{X,Y} (g,1,\Phi_{\overline {n} k})  dg d\overline {n} dk
      \end{equation*}
      of $J_{\Omega_{0,1},2} (\xi_{c,s})$ at $s=s_0$ does not vanish for some choice data. Here
      \begin{equation*}
        \Phi_{\overline {n}  k} = \omega_{X_1,Y} ( \overline {n} k, 1)\Phi(0,\cdot).
      \end{equation*}
\end{enumerate}
\end{prop}
\begin{proof}
  We just need to show part (3). We will suppress $\psi$ to save space. Since we assume that $\sigma\otimes \Theta_{X,Y}$ is $J (Z,L)$-distinguished, there exist $\phi\in\sigma$ and $\Psi\in\cS_{X,Y} (\A)$ such that
  \begin{equation*}
    \int_{[J (Z,L)]} \phi(h)\theta_{X,Y} ( h, 1, \Psi) dh
  \end{equation*}
  is non-vanishing.
  We may assume that $\Psi$ is a pure tensor of the form $\otimes_v \Psi_v$. We will construct $f_s\in\cA_1 (s,\chi,\sigma)$ from $\phi$ and $\Phi\in\cS_{X_1,Y} (\A)$ from $\Psi$. Let $S_\ur$ be the set of finite places of $F$ such that $\chi_v$ and $\psi_v$ are unramified, $\phi$ is invariant under $K_{G (X),v}$ and $\Psi_v$ is the characteristic function of $Y\otimes X^+ (\cO_v)$. Let $S_\rf$ be the rest of the  finite places of $F$. Let $S_\infty$ be the set of infinite places of $F$. Set $S_\ram = S_\rf \cup S_\infty$. The construction is much more involved than in Prop.~\ref{prop:non-van-of-key-period-L=0}.

  Assume that  $v\in S_\rf$. Choose a small compact open subgroup $\Omega'_v$ of $\overline {N}_{1,v}$ such that $\Omega_v := (Q_{1,v}\cap K_{G (X_1),v}) \cdot \Omega'_v \subset K_{G (X_1),v}$. For $k_v\in \Omega_v$ we write
  \begin{equation*}
    k_v = n_{k_v}m_1 (a_{k_v})h_{k_v} \omega'_{k_v}
  \end{equation*}
  with $n_{k_v}\in N_{1,v}\cap K_{G (X_1),v}$, $a_{k_v} \in \cO_v^\times$, $h_{k_v}\in G (X)_v \cap K_{G (X_1),v}$ and $\omega'_{k_v} \in \Omega'_v$. Assume  $v\in S_\infty$. Fix a complement of $\Lie (M_{1,v}\cap K_{G (X_1),v})$ in $\Lie (K_{G (X_1),v})$, choose a small open subset in this complement and let $\Omega'_v$ denote its exponential. Set $\Omega_v = (M_{1,v}\cap K_{G (X_1),v}) \cdot \Omega'_v$. In fact, this is $M_{1,v}\cap K_{G (X_1),v}$ times a small neighbourhood of the identity element. Also choose a smooth function $\alpha_v$ supported in $\Omega'_v$ with $\alpha_v (I)=1$.

  Construct a Schwartz function $\Psi_0 = \otimes_v \Psi_{0,v} \in \cS (Y\otimes Fe_1^+ (\A))$ and set $\Phi = \Psi_0 \otimes \Psi$. Assume $v\in S_\ur \cup S_\rf$. Take $\Psi_{0,v}\in \cS (Y\otimes Fe_1^+ (F_v))$ to be the characteristic function of $Y\otimes Fe_1^+ (\cO_v)$. For $v\in S_\rf$, we may shrink $\Omega'_v$ such that $\Phi_{0,v}$ is invariant under the action of $\Omega'_v$. For $v\in S_\infty$, take $\Psi_{0,v} \in \cS (Y\otimes Fe_1^+ (F_v))$ to be any Schwartz function taking value $1$ at the origin.

 Now  focus on the term $f_s(g\overline{n}k)$ in the period integral.  For an unramified place $v$, $f_s$ has to be $K_{G(X_1),v}$-invariant. For $v\in S_\ram$, we want $f_s|_{G(X_1)_v}$ to be supported in $\Omega_v$.  We check how elements in $NJ_1(\A)$ conjugate with elements in $N (W_1,Fe_1^-)(\A)$. This is to deal with the $n_{k_v}$ part of $k_v\in\Omega_v$. Let
  \begin{equation*}
    \overline {n} =
    \begin{pmatrix}
      1 &&&&\\
      a &1&&&\\
      &&I&&\\
      &&&1&\\
      &&&-a&1
    \end{pmatrix}
    \quad \text {and}\quad
     n =
    \begin{pmatrix}
      1 &0 &\mu &0 &\alpha\\
      &1 &&&0\\
      &&I&&\nu \\
      &&&1&0\\
      &&&&1
    \end{pmatrix}
  \end{equation*}
  with respect to the basis $e_1^+, f_1^+,W,f_1^-,e_1^-$.
  Then $\overline {n}n\overline {n}^{-1}$ is equal to
  \begin{equation*}
    \begin{pmatrix}
      1 & 0&\mu & a\alpha &\alpha\\
      &1  &a\mu &a^2\alpha &a\alpha\\
        &     &I   &a\nu &\nu\\
      &&&1&0\\
      &&&&1
    \end{pmatrix}
    = n' \cdot
    \begin{pmatrix}
      1 & 0&0 & 0 &0\\
      &1  &a\mu &a^2\alpha &0\\
        &     &I   &a\nu &0\\
      &&&1&0\\
      &&&&1
    \end{pmatrix}
  \end{equation*}
  for some $n'\in N_1 (\A)$. After conjugating with $g\in J (Z,L) (\A)$, $n'$ drops out of $f_s$. In addition, for the theta term in the period integral, $n'$ leaves $\Phi (0,\cdot)$ invariant. Also note that the matrix following $n'$ in the equation above will be absorbed into $J (Z,L) (\A)$ in the integration. Conjugation of $NJ_1 (\A)$ and $M_1 (W_1,Fe_1^-) (\A)$ is clear.

Now we check when $f_s(g\overline{n}k)$ can possibly be non-vanishing. Assume that $v\in S_\rf$. Recall that $w$ is the longest Weyl element in
  \begin{equation*}
    Q (W_1,Fe_1^-)_v\lmod G (W_1)_v /Q (W_1,Fe_1^-)_v.
  \end{equation*}
  Consider the Bruhat cells
    $Q (W_1,Fe_1^-)_vw' N (W_1,Fe_1^-)_vw$ 
  of $G (W_1)_v$. We are actually using a variant of Bruhat decomposition. The open cell is
  \begin{equation*}
    Q (W_1,Fe_1^-)_v \overline {N} (W_1,Fe_1^-)_v.
  \end{equation*}
  We can check that for $v\in S_\rf$, if $k_v\in K_{G (W_1),v}$ is not in the open Bruhat cell $Q (W_1,Fe_1^-)_v \overline {N} (W_1,Fe_1^-)_v$ then $\overline {n}_vk_v$ cannot be in $Q_{1,v}\overline {N}_{1,v}$. Thus for $f_s (g\overline {n}k)$ to not vanish, it is necessary that $k_v \in Q (W_1,Fe_1^-)_v \overline {N} (W_1,Fe_1^-)_v$. Suppose that $k_v \in Q (W_1,Fe_1^-)_v \overline {N} (W_1,Fe_1^-)_v\cap  K_{G (W_1),v}$ with decomposition $k_v = q_v \overline {n}'_v$. We have shown how $q_v$ conjugates over $\overline {n}_v$. Thus for $f_s (g\overline {n}k)$ to not vanish, it is necessary that $\overline {n}_v \overline {n}'_v \in \Omega'_v$. Since $NJ_{1,v}$ and $\overline {N} (W_1,Fe_1^-)_v$
  are `disjoint', the requirement is equivalent to $\overline {n}_v$ and $\overline {n}'_v \in \Omega'_v$. Thus just as in Prop.~\ref{prop:non-van-of-key-period-L=0}, only $k_v \in \Omega_v$ can contribute. In addition if we write $\overline {n}_v$ in matrix form as above, $a_v$ should have small $v$-adic norm.

  Assume that $v\in S_\infty$. Again assume $\overline {n}_v$ is given as above. Conjugating an element in $M (W_1,Fe_1^-)_v \cap K_{G (W_1)}$ across $\overline {n}_v$ will not change the norm of $a_v$. Thus for $f_s (g\overline {n}k)$ to be non-vanishing, it is necessary that $a_v \in F_v$ lies in a small open around $0$ and that $k_v \in \Omega_v$.

  Assume $v\in S_\ur$. If $a_v\not\in \cO_v$, we decompose $\overline {n}_v$ according to the Iwasawa decomposition induced from
    $\begin{pmatrix}
      1 & \\
      a_v &1
    \end{pmatrix} =
    \begin{pmatrix}
      a_v^{-1} &1 \\
       & a_v
    \end{pmatrix}
    \begin{pmatrix}
      & -1 \\
      1& a_v^{-1}
    \end{pmatrix}$.
  Since $f_s (\cdot k)$ for any $k\in K_{G (X_1)}$ is in the space of $\sigma$, thus by Lemma~\ref{lemma:rapid-dec-several-variable} for any $r_1$  there exists a constant $C_1$ such that we have bound
  \begin{equation*}
    |f_s (g\overline {n}k)| < C_1 \prod_{\substack{v\in S_\ur\\ a_v \not\in \cO_v}} |a_v^{-1}|_v^{s+\rho_{Q_1}} |a_v|_v^{-r_1\dim X} m_{P_0} (g)^{-r_1\rho_{P_0}}.
  \end{equation*}
  We have made use of the fact that for $v\in S_\ram$, $a_v$ has small norm.
  Also we have bound for theta series
  \begin{equation*}
    |\theta_{X,Y} (g,1,\Phi_{\overline {n}k})| < C_2 ||g||^{r_2}||\overline {n}||^{r_2}.
  \end{equation*}
  for some constant $C_2$ and $r_2>0$.
  Note that
  \begin{equation*}
    ||\overline {n}|| = \prod_{v} \sup\{|a_v|_v, 1\} = \prod_{\substack{v\in S_\ur\\ a_v \not\in \cO_v}} |a_v|_v.
  \end{equation*}
  Here again we have used that fact that for $v\in S_\ram$, $a_v$ has small norm.
  Fix $s$ and fix a large enough $r_1$. Then for any $0<\varepsilon <1$, there exists $C_3$ such that if norm of $v\in S_\ur$ is larger than $C_3$ then
  \begin{equation*}
    |a_v^{-1}|_v^{s+\rho_{Q_1}} |a_v|_v^{-r_1\dim X + r_2} < \varepsilon.
  \end{equation*}
  We shrink $S_\ur$ to exclude these finitely many places of small norm and expand $S_\rf$, so that the above bound holds for all $v\in S_\ur$. We remark that $C_3$ depends on $C_1$ and $C_2$. However we may use the same $C_1$ in the bound for  $f_s$  even though it is changed after enlarging the set $S_\rf$.
  Also enlarging $S_\rf$ has no effect on the $\Phi$ we have constructed. In short, $C_3$ depends only on $\phi\in\sigma$ and $\Psi\in\cS_{X,Y}(\A)$.  Then we find that
  \begin{equation*}
    \int_{K_{G (W_1)}}\int_{\substack{NJ_1 (\A)\\ a_v\not\in \cO_v\\ \text {for some } v\in S_\ur} }\int_{[J (Z,L)]}     |f_s (g\overline {n}k)\theta_{X,Y} (g,1,\Phi_{\overline {n}k})| dg d\overline {n} dk   < C_4 \varepsilon
  \end{equation*}
  where $C_4 = \int_{[J (Z,L)]} C_1 C_2 m_{P_0} (g)^{-r_1\rho_{P_0}} ||g||^{r_2} dg$ is a finite number. Thus this part is negligible.
  Finally we check that the remaining part is equal to
  \begin{align*}
    & \int_{K_{G (W_1)}}\int_{\substack{\prod_{v\in S_\ram} NJ_1 (F_v) \times \prod_{v\in S_\ur} NJ_1 (\cO_v)} }\int_{[J (Z,L)]}     f_s (g\overline {n}k)\theta_{X,Y} (g,1,\Phi_{\overline {n}k}) dg d\overline {n} dk   \\
    =& \int_{\prod_{v\in S_\infty} \Omega'_v}\int_{\substack{ NJ_{1,\infty} \\ a_\infty \text {has small norm} } }\int_{[J (Z,L)]}     f_s (g\overline {n}k)\theta_{X,Y} (g,1,\Phi_{\overline {n}k}) dg d\overline {n} dk
  \end{align*}
  where we may need to shrink $\Omega'_v$ further to get invariance for theta series to get the equality.
  By continuity, for small enough $\Omega'_v$ with $v\in S_\infty$, the above is non-vanishing.
\end{proof}


Now we compute the integral \eqref{eq:J-Omega01-3}.
We pick $\delta_{Ff_1^-}$ to be the element in $G (Z_1)$ that is
determined by the following action:
$e_1^+ \leftrightarrow f_1^+$ and $ e_1^- \leftrightarrow f_1^-$
with identity action on orthogonal complement.  Then
$\delta_{Ff_1^-} J (Z_1,L) \delta_{Ff_1^-}^{-1} = J
(Z_1,Fe_1^-)$. Thus \eqref{eq:J-Omega01-3} is equal to
\begin{equation*}
 \int_{[J (Z_1,Fe_1^-)]} \xi ( g \delta_{Ff_1^-}) dg
=\int_{[N (Z_1, Fe_1^-)]} \int_{[G (Z)]} \xi (n g \delta_{Ff_1^-})
dgdn.
\end{equation*}
We see that the above has an inner integral of the form
$\int_{[G (Z)]} \xi ( g n \delta_{Ff_1^-})dg$.

\begin{prop}\label{prop:van-J-Omega-01-3}
  If $\sigma\otimes \Theta_{\psi,X,Y}$ is not $G (Z)$-distinguished, then $J_{\Omega_{0,1},3} (\xi)$ vanishes for $\xi = \xi_{c,s}$ and $\xi_s^c$.
\end{prop}

\subsubsection{$I_{\Omega_{1,0}}$}
We need to consider
\eqref{eq:I-Omega10}.
For each $v\in V-\{0\}$, we choose $v^+\in V$ that is dual to $v$
so that $\form {v^+} {v}_V=1$. We pick $\gamma_{Fv}$ to be the element
in $G (X_1)$ that is determined by the following action:
$e_1^+ \leftrightarrow v^+$ and $ e_1^- \leftrightarrow v$,
with identity action on orthogonal complement. Then
$\gamma_{Fv}J (Z_1,L)\gamma_{Fv}^{-1} = J (Fv^+\oplus Z \oplus Fv , L)
$. Thus \eqref{eq:I-Omega10} is equal to
\begin{equation*} \sum_{v\in V-\{0\} / F^\times} \int_{[J (Fv^+\oplus
Z \oplus Fv,L)]} \xi (g \gamma_{Fv} )dg
\end{equation*}
and we get:

\begin{prop}\label{prop:van-I-Omega-10-L=1}
  If $\sigma\otimes \Theta_{\psi,X,Y}$ is not $J (Z', L)$-distinguished for all $Z' \supset Z$ such that $\dim Z' = \dim Z +2$, then $I_{\Omega_{1,0}} (\xi)$ vanishes for $\xi = \xi_{c,s}$ and $\xi_s^c$.
\end{prop}

\subsubsection{$I_{\Omega_{1,1}}$}
We need to consider \eqref{eq:I-Omega11}.
For each fixed $v\in V-\{0\} / F^\times$, we analyse each term:
$\int_{[J (Z_1,L)]} \sum_{\delta\in J (Z_1,e_1^-)
\lmod G (Z_1)} \xi (\gamma_{F (v+e_1^-)}\delta g)dg$.

We choose $v^+\in V$, dual to $v$, so that
$\form {v^+} {v}_V=1$. The set $J (Z_1,Fe_1^-) \lmod G (Z_1)$
parametrises non-zero isotropic vectors in $Z_1$. We will make use of
notation we set up for $I_{\Omega_{0,1}}$. For a non-zero isotropic
vector $z \in Z_1$ we set $\delta_z$ to be an element in $G (Z_1)$
such that $e_1^-\delta_v = v$. We consider their $J
(Z_1,L)$-orbits. An isotropic vector $ af_1^+ + w + bf_1^-$ for
$a,b\in F$ and $w\in W_1$ is in the same $J (Z_1,L)$-orbit as
\begin{align*} af_1^+, \quad &\text{if $a\neq 0$};\\ e_1^-, \quad
&\text{if $a= 0$ and $w\neq 0$};\\ bf_1^-, \quad &\text{if $a= 0$ and
$w= 0$}.
\end{align*}
For $a\in F^\times$ the stabiliser of $af_1^+$ in $J
(Z_1,L)$ is $G (W_1)$ and hence
\begin{equation*}
  \Stab_{J (Z_1,L)} (af_1^+) \lmod J(Z_1,L) \isom N (Z_1,L).
\end{equation*}
For the isotropic vector $e_1^-$ we have
\begin{equation*}
 \Stab_{J (Z_1,L)} (e_1^-) \isom NJ_2 \rtimes J (W_1,Fe_1^-).
\end{equation*}
For $bf_1^-$ with $b\in F^\times$, the stabiliser $\Stab_{J (Z_1,L)}
(bf_1^-) $ is $J (Z_1,L)$. Thus we further split $I_{\Omega_{1,1}}$ as
$J_{\Omega_{1,1},1} + J_{\Omega_{1,1},2} +J_{\Omega_{1,1},3}$ where
\begin{align}
  \label{eq:J-Omega11-1} J_{\Omega_{1,1},1} (\xi) =&\int_{[J (Z_1,L)]}
\sum_{a\in F^\times}\sum_{\eta \in N (Z_1, L)} \xi (\gamma_{F
(v+e_1^-)}\delta_{af_1^+}\eta g) dg;\\
\label{eq:J-Omega11-2} J_{\Omega_{1,1},2} (\xi) =&\int_{[J (Z_1,L)]}
\sum_{\eta \in NJ_2 \rtimes J (W_1,Fe_1^-) \lmod J (Z_1,L)} \xi
(\gamma_{F (v+e_1^-)} \eta g) dg;\\
\label{eq:J-Omega11-3} J_{\Omega_{1,1},3} (\xi) =&\int_{[J (Z_1,L)]}
\sum_{b\in F^\times} \xi (\gamma_{F (v+e_1^-)} \delta_{bf_1^-} g) dg.
\end{align} We recall that we have chosen $\gamma_{F (v+e_1^-)}$ to be
$\gamma_1\gamma_2$ as in \eqref{eq:gamma_F-v+e1-=gamma1gamma2}.

 We find that \eqref{eq:J-Omega11-1} is equal to
\begin{align*} & \int_{[J (Z_1,L)]} \sum_{a\in F^\times}\sum_{\eta \in
N (Z_1, L)} \xi (\gamma_{F (v+e_1^-)}\delta_{af_1^+}\eta g) dg\\ =&
\int_{N (Z_1,L) (\A)} \int_{[G (W_1)]} \sum_{a\in F^\times} \xi
(\gamma_{F (v+e_1^-)}\delta_{af_1^+} gn) dgdn.
\end{align*} We note that $\delta_{af_1^+} G (W_1)
\delta_{af_1^+}^{-1} = G (Z)$ and $\gamma_{F (v+e_1^-)}$ commutes with
elements in $G (Z)$. Thus the above is equal to
\begin{align*} \int_{N (Z_1,L) (\A)} \int_{[G (Z)]} \sum_{a\in
F^\times} \xi (g \gamma_{F (v+e_1^-)}\delta_{af_1^+} n) dgdn.
\end{align*}

\begin{prop}\label{prop:van-J-Omega-11-1}
  If $\sigma\otimes \Theta_{\psi,X,Y}$ is not $G (Z)$-distinguished, then $J_{\Omega_{1,1},1} (\xi)$ vanishes for $\xi = \xi_{c,s}$ and $\xi_s^c$.
\end{prop}

 The integral \eqref{eq:J-Omega11-2} is equal to
\begin{align*} &\int_{[J (Z_1,L)]} \sum_{\eta \in NJ_2 \rtimes J
(W_1,Fe_1^-) \lmod J (Z_1,L)} \xi (\gamma_{F (v+e_1^-)} \eta g) dg\\
=&\int_{NJ_1 (\A)} \int_{[NJ_2]} \int_{J (W_1,Fe_1^-) \lmod G (W_1)
(\A)} \xi (\gamma_{F (v+e_1^-)} n_1 n_2 g ) dg dn_2 dn_1\\
=&\int_{NJ_1 (\A)} \int_{K_{G (W_1)}}\int_{\A^\times}
\int_{[NJ_2]}\int_{[J (W_1,Fe_1^-)] } \xi (\gamma_{F (v+e_1^-)}n_1 n_2
g m_1 (t)k ) dgdn_2 dt dk dn_1\\ =&\int_{NJ_1 (\A)} \int_{K_{G
(W_1)}}\int_{\A^\times} \int_{[J (Z_1,Fe_1^-\oplus Ff_1^-)] } \xi
(\gamma_{F (v+e_1^-)} n_1 g m_1 (t)k ) dg dt dk dn_1\\
=&\int_{NJ_1 (\A)} \int_{K_{G (W_1)}}\int_{\A^\times} \int_{[J
(Z_1,Fe_1^-\oplus Ff_1^-)] } \xi (\gamma_{F (v+e_1^-)} gn_1 m_1 (t)k )
dg dt dk dn_1.
\end{align*} We can check that $\gamma_2$ commutes with $g\in J
(Z_1,Fe_1^-\oplus Ff_1^-)$ and that
\begin{equation*}
\gamma_1J (Z_1,Fe_1^-\oplus
Ff_1^-)\gamma_1^{-1} = J (Fv^+\oplus Z \oplus Fv,Fv\oplus
Ff_1^-).
\end{equation*}
 Thus we get an inner integral:
$\int_{[J (Fv^+\oplus Z \oplus Fv,Fv\oplus Ff_1^-)] }
\xi ( g\gamma_{F (v+e_1^-)}n_1 m_1 (t)k) dg$.

\begin{prop}\label{prop:van-J-Omega-11-2}
  If $\sigma\otimes \Theta_{\psi,X,Y}$ is not $J (Z',L'\oplus L)$-distinguished where $Z'\supset Z$ such that $\dim Z' = \dim Z+2$ and $L'$ is an isotropic line of $Z'$ in the orthogonal complement of $Z$, then $J_{\Omega_{1,1},2} (\xi)$ vanishes for $\xi = \xi_{c,s}$ and $\xi_s^c$.
\end{prop}

 For the integral \eqref{eq:J-Omega11-3} we note that $\delta_{bf_1^-}
J (Z_1,L)\delta_{bf_1^-}^{-1} = J (Z_1,Fe_1^-)$, that $\gamma_2$
commutes with elements in $J (Z_1,Fe_1^-)$ and that
\begin{equation*}
\gamma_1 J
(Z_1,Fe_1^-)\gamma_1^{-1} = J (Fv^+\oplus Z \oplus Fv, Fv).
\end{equation*}
 Thus we get
$\int_{[J (Fv^+\oplus Z \oplus Fv, Fv)]} \sum_{b\in
F^\times} \xi (g\gamma_{F (v+e_1^-)} \delta_{bf_1^-} ) dg$.

\begin{prop}\label{prop:van-J-Omega-11-3}
  If $\sigma\otimes \Theta_{\psi,X,Y}$ is not $J (Z',L')$-distinguished where $Z'\supset Z$ such that $\dim Z' = \dim Z+2$ and $L'$ is an isotropic line of $Z'$ in the orthogonal complement of $Z$, then $J_{\Omega_{1,1},3} (\xi)$ vanishes for $\xi = \xi_{c,s}$ and $\xi_s^c$.
\end{prop}

\section{On Absolute Convergence of Integrals}
\label{sec:abs-conv-issue}

In this section we prove Prop.~\ref{prop:abs-conv} which concerns absolute convergence of various integrals to justify our computation. As  absolute convergence for the case $\dim L =1$ follows from the case where $\dim L = 0$, we just treat the latter case. Thus throughout this section the integrals are over $[G (Z_1)]$.

\subsection{Bounds for cuspidal automorphic forms and theta series}
\label{sec:bounds}

We need some bounds in the later computation and we
list the results here.  Let $Z$
and $V$ be two non-degenerate subspaces of $X$  orthogonal to each
other and assume $X= Z \oplus V$. Choose minimal parabolic subgroup $P_{0,Z}$ (resp. $P_0$) of $G (Z)$ (resp. $G (X)$) such that
$P_{0,Z} = P_0 \cap G (Z)$. Let $M_0$ denote the Levi subgroup of $P_0$ and $N_0$ the unipotent radical. Set
$M_0 (\A)^1 = \cap_{\chi \in \Rat (M_0)} \Ker |\chi|$, 
where $\Rat (M_0)$ denotes the set of all rational characters of $M_0$.
Recall the map
\begin{equation*}
  m_{P_0}: G (X) (\A) \rightarrow M_0 (\A)^1 \lmod M_0 (\A)
\end{equation*}
as defined in \cite [I.1.4] {MR1361168}. Write $g\in G (X) (\A)$ as $g=umk$ for $u\in N_0 (\A)$, $m\in M_0 (\A)$ and $k\in K_{G (X)}$. Then $m_{P_0} (g)= M_0 (\A)^1m$. This is well-defined. We also define the analogous map $m_{P_{0,Z}}$ for $G (Z) (\A)$.

\begin{lemma}\label{lemma:rapid-dec-several-variable} Let $\sigma \in
\cA_\cusp (G (X))$ and $\phi\in \sigma$.  Then for all positive
integer $r$, there exists a constant $C$ such that
\begin{equation*} | \phi ( h g) | < C m_{P_0} (h)^{-r \rho_{P_0}}
  \end{equation*} for all $h\in G (Z) (\A)$ and $g\in G( V )(\A)$.
\end{lemma}
\begin{proof}
  Assume that $P_{0,Z}$ stabilises  the complete isotropic flag $Z^{(1)} \subset \cdots \subset Z^{(t)}$ of  $Z$. Let $P_1$ denote the parabolic subgroup of $G (X)$ stabilising the isotropic flag $Z^{(1)} \subset \cdots \subset Z^{(t)}$.
  We note that $M_{P_1} (\A)^1m_{P_0} (h) = M_{P_1} (\A)^1 m_{P_1} (hg)$ and $\rho_{P_0} = \rho_{P_1}$ when restricted to the image of $M_{0,Z} (\A)$ in $M_0 (\A)$. Then by \cite [Thm. D] {MR2955198}, for
all positive integer $r$, there exists a constant $C$ such that
$| \phi ( h g) | < C m_{P_0} (h)^{-r \rho_{P_0}}$. 
\end{proof}

Let $v_0\in V$ and let $v_0^+\in V$ be a dual vector to $v_0$. Let $U$ be the orthogonal complement of $Fv_0^+\oplus Fv_0$ in $V$. Here we assume that $P_0$ is compatible with $P_{0,Fv_0^+\oplus Z \oplus Fv_0}$.
\begin{lemma}\label{lemma:rapid-dec-with-n} Let $\sigma \in \cA_\cusp
(G (X))$ and $\phi\in \sigma$.  Then for all positive integer $r$
there exists a constant $C$ such that
\begin{equation*} | \phi (n h g) | < C m_{P_0}(h)^{-r\rho_{P_0}}
  \end{equation*} for all $h\in G (Fv_0^+\oplus Z \oplus Fv_0) (\A)$, $g\in G (U) (\A)$ and $n \in
N (X,Fv_0) (\A)$.
\end{lemma}
\begin{proof}
Assume that $P_{0,Z}$ stabilises  the complete isotropic flag $Z^{(1)} \subset \cdots \subset Z^{(t)}$ of  $Z$. Let $P_2$ denote the parabolic subgroup of $G (X)$ stabilising the isotropic flag $Fv_0\subset Z^{(1)}\oplus Fv_0 \subset \cdots \subset Z^{(t)}\oplus Fv_0$.
Then $M_{P_2} (\A)^1m_{P_0} (h) = M_{P_2} (\A)^1 m_{P_2} (nhg)$ and $\rho_{P_0} = \rho_{P_2}$ when restricted to the image of $M_{0,Fv_0^+\oplus Z \oplus Fv_0} (\A)$ in $M_0 (\A)$. Thus by
by \cite [Thm. D] {MR2955198}, for
all positive integer $r$, there exists a constant $C$ such that
$| \phi (n h g) | < C  m_{P_0}(h)^{-r\rho_{P_0}}$. 
\end{proof}

Finally we need an estimate of theta series. For $h\in G (X) (\A)$, let $||h||$ denote the height of $h$ as defined in \cite [I.2.2] {MR1361168}.
\begin{lemma}\label{lemma:theta-bound} Fix $\Phi \in \cS_{X_1,Y}
(\A)$.  There exist $C>0$, $r>0$ and $R>0$ such that
  \begin{equation*} |\theta_{X_1,Y} (n m_1 (t)h k,1,\Phi)| < C
(|t|_{\A}^{\frac{\dim Y}{2}}+ |t|_{\A}^{-r}) || h ||^R
  \end{equation*} for all $n \in N_{1} (\A)$, $t\in \GL_1 (\A)$, $h\in
  G (X) (\A)$ and $k\in K_{G (X_1)}$.
\end{lemma}
\begin{proof} First we fix a basis for $Y$ and a symplectic basis for
$X$. Thus we have a polarisation: $X = X^+ \oplus X^-$. The Schwartz
space $\cS_{X_1,Y} (\A)$ is realised on the space $\cS(Y\otimes \ell_1^+
(\A) \oplus Y\otimes X^+ (\A))$.  We may assume that $n$ lies in a
compact set because the theta series is invariant under $N_1 (F)$ on the left. Then by continuity we only need to show the bound for
fixed $n$.  Furthermore via reduction theory we may assume that $h$
lies in the Siegel domain $\omega A_{0,G (X)} (t_0) K_{G (X)}$ of $G
(X)$. We refer the reader to \cite [I.2.1] {MR1361168} for
notation. The only difference is that we add in subscripts to indicate
the group under consideration. Since
$\cup_{a\in A_{0,G (X)} (t_0)} a^{-1}\omega a$ 
is relatively compact in $P_0 (\A)$, we may
assume further that $h\in A_{0,G (X)}$. We may also assume that $t$ lies in
$\A_\infty^\times$ rather than $\A^\times$. It is not essential if $h$ lies in $A_{0,G (X)} (t_0)$ or the bigger set $A_{0,G (X)}$.
Let $\kappa = [F:\Q]$. Given $a\in \R_{>0}$ we let $\tau_a$ be
the element in $\A_\infty^\times$ which is equal to $a^{1/\kappa}$ at any
infinite place. Write $h = m (\tau_{a_1},\ldots,\tau_{ a_r})$ for
$a_i\in \R_{>0}$ where $r=\dim X/2$. Let $a=\diag \{a_1,\ldots, a_r\}$
and $\tau_a =\diag \{\tau_{a_1},\ldots,\tau_{ a_r}\}$. We may replace $t$
by $\tau_{t}$ with $t\in \R_{>0}$. Thus we are reduced to showing that there
exists $C>0$, $a>0$ and $R>0$ such that
\begin{equation*} |\theta_{X_1,Y} (n m_1 (\tau_t)h k,1,\Phi)| < C
(|t|^{\frac{\dim Y}{2}}+ |t|^{-r}) || h ||^R
\end{equation*} for fixed $n$, for all $t\in \R_{>0}$ and $ h=m (\tau_{a_1},\ldots,\tau_{ a_r})\in A_{0,G (X)}$ and  $k\in
K'_{G (X_1)}$ where $K'_{G (X_1)}$ is some compact subset of $G (X_1) (\A)$.

Via explicit formulae of Weil representation, we find
\begin{align*} &|\theta_{X_1,Y} (n m_1 (\tau_t)h k,1,\Phi)| =
|\sum_{\substack{ y\in Y\otimes \ell_1^+ \\ w \in Y\otimes X^+ }}
\omega_{X_1,Y} (n m_1 (\tau_t)h k,1)\Phi (y,w)| \\ \le &
\sum_{\substack { y\in Y\otimes \ell_1^+ \\ w \in Y\otimes X^+}}
|t\det a|^{\dim Y/2}|\omega_{X_1,Y} ( k,1)\Phi (y\tau_t,(y\mu +
w)\tau_a)| ,
\end{align*} where $\mu \in (\Hom_F (\ell_1^+,X)) (\A)$ is determined
by $n$ and can be viewed as a row vector of length $\dim X$.  By \cite
[Lemme~5] {MR0165033} there exists $\Phi_0\in \cS_{X_1,Y} (\A)$ such
that
\begin{align*} |\omega_{X_1,Y} ( k,1)\Phi (y,w)| \le \Phi_0 (y,w)
\end{align*} for all $k\in K'_{G (X_1)}$, $y\in Y\otimes \ell_1^+ (\A)$
and $ w \in Y\otimes X^+ (\A)$. Henceforth we replace $\Phi$ by
$\Phi_0$ but still write it as $\Phi$.  The problem reduces to finding
a bound for
\begin{equation}\label{eq:sum-theta} \sum_{\substack { y\in Y\otimes
\ell_1^+ \\ w \in Y\otimes X^+}} \Phi (y\tau_t,(y\mu + w)\tau_a).
\end{equation}
We may assume that $\Phi$ can be separated into archimedean and finite part: $\Phi = \Phi_\infty \otimes \Phi_f$
with $\Phi_f$ being the characteristic function of some open compact
subgroup $C_f^{(1)} \times C_f^{(2)}$ of $Y\otimes \ell_1^+ (\A_f)
\oplus Y\otimes X^+ (\A_f)$. This open compact subgroup determines a
lattice $\Lambda_1\oplus \Lambda_2$ in $Y\otimes \ell_1^+ \oplus
Y\otimes X^+ $. Furthermore we may assume that after suitably
enlarging $C_f^{(2)}$, for all $y_f\in C_f^{(1)} $, we have $y_f\mu_f \in
C_f^{(2)}$ where $\mu_f$ is the finite part of $\mu$. Then \eqref{eq:sum-theta} is equal to
\begin{equation}\label{eq:sum-theta-2} \sum_{\substack { y\in
\Lambda_1 \\ w \in \Lambda_2}} \Phi_\infty (y\tau_t,(y\mu_{\infty} +
w)\tau_a).
\end{equation} We adapt the basis for $ Y\otimes \ell_1^+ \oplus
Y\otimes X^+$ to a $\Q$-basis of $\Res_{F/\Q} ( Y\otimes \ell_1^+
\oplus Y\otimes X^+)$. In fact we just need to consider $Y$ as a
$\kappa\dim_F Y$-dimensional $\Q$-vector space.  After scaling we can
make $\Lambda_1$ and $\Lambda_2$ integral with respect to the basis.
Then in our case, the bound in \cite [P.~21] {MR0223373} reads
\begin{equation*} \Phi_\infty (y,w) \le C_\alpha \prod_{j=1}^{\kappa
\dim Y} \prod_{i=1}^{\dim X/2} (1 + |y_j|^\alpha)^{-1} (1+
|w_{ij}|^\alpha)^{-1}.
\end{equation*} for $\alpha >1$. Since we have fixed a basis, $y$ and $w$ can be viewed as matrices and the subscripts indicate the corresponding entry. The linear map $\mu$  also becomes a matrix (or rather a vector). Thus \eqref{eq:sum-theta-2} is
bounded by
\begin{align*} &C_\alpha \prod_{j=1}^{\kappa \dim Y} \prod_{i=1}^{\dim
X/2} \sum_{m=-\infty}^{\infty} \sum_{n=-\infty}^{\infty} (1 +
|t^{1/\kappa} n|^\alpha)^{-1} (1+ | a_i^{1/\kappa}(n\mu_{\infty,j} +
m) |^\alpha)^{-1} \\ \le & C'_{\alpha} \sup \{1 , t^{-\alpha}\}^{\dim
Y} \prod_{i=1}^{\dim X/2}\sup \{1,a_i^{-\alpha}\}^{\dim Y}.
\end{align*} Thus the theta series is bounded by the above multiplied
by $|t\det a|^{\dim Y/2}$. We see that the $t$-part of the bound is bounded
by $|t|_{\A}^{\frac{\dim Y}{2}}+ |t|_{\A}^{-c}$ and that the $h$-part of
the bound is bounded by
$\prod_{i=1}^{\dim X/2}\sup \{a_i^{\dim
Y/2},a_i^{-c}\}$
for some $c>0$. From the definition of  height $||h||$ of $h$, we get our desired
bound.
\end{proof}

\subsection{Absolute Convergence of $I_{\Omega_{0,1}}$}
\label{sec:I-Omega01-abs-conv}

In order to show absolute convergence of $I_{\Omega_{0,1}} (\xi)$ (c.f. \eqref{eq:I-Omega01}), we just need
to show absolute convergence of
\eqref{eq:I-Omega01-for-abs-conv}. Assume $\xi=\xi_{c,s}$. Then using
Lemma~\ref{lemma:theta-bound} and \cite
[Thm. A] {MR2955198}, we find that for all $n>0$ there exists a constant $C$ such that the integrand is bounded by
\begin{equation*} C(|t|_\A^{\dim Y/2} + |t|_\A^{-r}) || h ||^{-n}
|t|_\A^{s - 2\rho_{Q (Z_1,Fe_1^-)} + \rho_{Q_1}} \hat {\tau}_c (m_1
(t))
\end{equation*} for some $r>0$ and all $h$ in a Siegel
domain of $G (Z_1) (\A)$. Thus integration over $h\in [G (Z)]$
converges.  Thanks to truncation that gives upper bound to $|t|_\A$,
integration over $t$ also converges as long as $\Re s$ is large enough
to make real part of exponents of $|t|_\A$ positive. Integration over $n \in [N
(Z_1,Fe_1^-)]$ and $k\in K_{G (Z_1)}$ is compact integration. Thus we
have shown that $I_{\Omega_{0,1}} (\xi_{c,s})$ is absolutely
convergent for $\Re s $ large enough. Similarly we can check absolute
convergence of $I_{\Omega_{0,1}} (\xi_s^c)$ for $\Re s $ large enough.

\begin{lemma}
  For $\Re s $ large enough, $I_{\Omega_{0,1}} (\xi_{c,s})$ and $I_{\Omega_{0,1}} (\xi_s^c)$ are absolutely convergent.
\end{lemma}

\subsection{Absolute Convergence of $I_{\Omega_{1,0}}$}
\label{sec:I-Omega10-abs-conv}

Next we consider absolute convergence of $I_{\Omega_{1,0}} (\xi)$ (c.f.  \eqref{eq:I-Omega10}). We
need to make our choice of $\gamma_{Fv}$'s more uniform. Fix $v_0\in V$
and a dual vector $v_0^+ \in V$. Let $U$ be the subspace of $V$ such
that $V=Fv_0^+ \oplus U \oplus Fv_0$. The set $V-\{0\} / F^\times$ is
parametrised by $Q (V,Fv_0)\lmod G (V)$. We use the correspondence: $\eta \mapsto Fv_0\eta$ between 
$Q (V,Fv_0)\lmod G (V)$ and $V-\{0\} / F^\times$.
For $v\in V-\{0\}$ we write $\eta_{Fv}$ for its preimage.
Fix $\gamma_{Fv_0}$ as in Sec.~\ref{sec:I-Omega10}. Note that it is in
$K _{G (X_1)}$. Then we may pick $\gamma_{Fv}$ to be
$\gamma_{Fv_0}\eta_{Fv}$. We rewrite $I_{\Omega_{1,0}}$ as
 \begin{align*} \int_{[G(Z_1)]} \sum_{\eta \in Q (V,Fv_0)\lmod G (V)}
\xi (\gamma_{Fv_0}\eta g)dg.
 \end{align*} Assume that $\xi=\xi_{c,s}$. In fact the theta part is
invariant under $G (X_1) (F)$. Thus for $g\in G (Z_1) (\A)$, we have
\begin{align*} \xi_{c,s} (\gamma_{Fv_0}\eta g) = &f_s
(\gamma_{Fv_0}\eta g) \theta_{X_1,Y} (\gamma_{Fv_0}g,1,\Phi)\hat
{\tau}_c (H (\gamma_{Fv_0}\eta g))\\ = &f_s (\gamma_{Fv_0}g \eta)
\theta_{X_1,Y} (\gamma_{Fv_0}g,1,\Phi)\hat {\tau}_c (H
(\gamma_{Fv_0}g\eta )).
\end{align*} Conjugate $\gamma_{Fv_0}$ across $g$ we get an integral
\begin{align*} \int_{[G(Fv_0^+ \oplus Z \oplus Fv_0)]} \sum_{\eta \in
Q (V,Fv_0)\lmod G (V)} f_s (g \gamma_{Fv_0} \eta) \theta_{X_1,Y} (g
\gamma_{Fv_0},1,\Phi)\hat {\tau}_c (H (\gamma_{Fv_0} \eta ))dg.
 \end{align*} We use $\{1\} \sqcup w_0N (V, Fv_0)$ as representatives
for $Q (V,Fv_0)\lmod G (V)$ where $w_0$ is the image in $G (V)$ of the
element that is
  $\smatrix{0}{1}{-1}{0}$,
with respect to the basis $v_0^+, v_0$. For $\eta = 1
$, the integral is absolutely convergent. We just need to show
absolute convergence for
\begin{align*} &\int_{[G(Fv_0^+ \oplus Z \oplus Fv_0)]} \sum_{\eta \in
N (V,Fv_0)} f_s (g \gamma_{Fv_0} w_0 \eta) \theta_{X_1,Y} (g
\gamma_{Fv_0},1,\Phi)\hat {\tau}_c (H (\gamma_{Fv_0} w_0 \eta ))dg \\
=&\int_{[G(Fv_0^+ \oplus Z \oplus Fv_0)]} \sum_{\overline {\eta} \in
\overline {N} (V,Fv_0)} f_s (g \gamma_{Fv_0} \overline {\eta} w_0)
\theta_{X_1,Y} (g \gamma_{Fv_0},1,\Phi)\hat {\tau}_c (H (\gamma_{Fv_0}
\overline {\eta} ))dg \\ =&\int_{[G(Fv_0^+ \oplus Z \oplus Fv_0)]}
\sum_{\overline {\eta} \in \overline {N} (U_1,Fe_1^-)} f_s (\overline
{\eta}g\gamma_{Fv_0} w_0) \theta_{X_1,Y} (g \gamma_{Fv_0},1,\Phi)\hat
{\tau}_c (H (\overline {\eta}))dg .
\end{align*} Write $\overline {\eta} = n_{\overline {\eta} } m_1
(a_{\overline {\eta} }) h_{\overline {\eta} } k_{\overline {\eta} }$
in the Iwasawa decomposition with $n_{\overline {\eta} } \in N
(U_1,Fe_1^-) (\A)$, $a_{\overline {\eta} }\in \GL_1 (\A)$,
$h_{\overline {\eta} }\in G (U (\A))$ and $k_{\overline {\eta} } \in
K_{G (U_1)}$. Then the above is equal to
\begin{align*} &\int_{[G(Fv_0^+ \oplus Z \oplus Fv_0)]}
\sum_{\overline {\eta} \in \overline {N} (U_1,Fe_1^-)} |a_{\overline
{\eta} }|_\A^{s+\rho_{Q_1}} f_s (h_{\overline {\eta} }k_{\overline
{\eta}} g\gamma_{Fv_0} w_0) \theta_{X_1,Y} (g
\gamma_{Fv_0},1,\Phi)\hat {\tau}_c (H ( m_1 (a_{\overline {\eta}
})))dg \\ =&\int_{[G(Fv_0^+ \oplus Z \oplus Fv_0)]} \sum_{\overline
{\eta} \in \overline {N} (U_1,Fe_1^-)} |a_{\overline {\eta}
}|_\A^{s+\rho_{Q_1}} f_s ( gh_{\overline {\eta} }k_{\overline
{\eta}}\gamma_{Fv_0} w_0) \theta_{X_1,Y} (g \gamma_{Fv_0},1,\Phi)\hat
{\tau}_c (H ( m_1 (a_{\overline {\eta} })))dg .
\end{align*}
By Lemma~\ref{lemma:rapid-dec-several-variable}, we have a uniform bound
for $f_s (g h_{\overline {\eta}}k_{\overline {\eta}}\gamma_{Fv_0} w_0)$ which is in the space of the
cuspidal representation $\sigma$, i.e., for all $r_1>1$ there exists
$C_1>0$ such that
 \begin{align*} |f_s ( g h_{\overline {\eta} }k_{\overline
   {\eta}}\gamma_{Fv_0} w_0)| < C_1  m_{P_0} (g)^{-r_1\rho_{P_0}}
 \end{align*} for all $g \in G (Fv_0^+ \oplus Z
\oplus Fv_0) (\A)$ and all $\overline {\eta} \in \overline {N}
(U_1,Fe_1^-)$. Also $\theta_{X_1,Y}$ is
of moderate growth in $g$, i.e., there exist $C_2, r_2 \in \R$ such that
\begin{align*} |\theta_{X_1,Y} (g \gamma_{Fv_0},1,\Phi)| < C_2 ||g||^{r_2}
\end{align*} for all $g\in G (Fv_0^+ \oplus Z \oplus Fv_0) (\A)$. Thus
integration over $[G(Fv_0^+ \oplus Z \oplus Fv_0)]$ is absolutely
convergent. It remains to show that
$$\sum_{\overline {\eta} \in \overline {N} (U_1,Fe_1^-)}
|a_{\overline {\eta} }|_\A^{s+\rho_{Q_1}}\hat {\tau}_c (H ( m_1
(a_{\overline {\eta} })))$$
is convergent or a fortiori
$\sum_{\overline {\eta} \in \overline {N} (U_1,Fe_1^-)}
|a_{\overline {\eta} }|_\A^{s+\rho_{Q_1}}$ is.  We find a bound for this series. It is equal to
$\sum_{ \eta \in \overline {N} (U_1,Fe_1^-) w_0}
|a_{\eta }|_\A^{s+\rho_{Q_1}}$,
which is bounded above by
\begin{equation*}
\sum_{ \eta \in \{1\} \sqcup \overline {N}
(U_1,Fe_1^-) w_0} |a_{\eta}|_\A^{s+\rho_{Q_1}} =\sum_{ \eta \in Q
(U_1,Fe_1^-)\lmod G (U_1)} |a_{\eta}|_\A^{s+\rho_{Q_1}}.
\end{equation*}
The last one is an Eisenstein series. Thus for $\Re s$ large
enough this is convergent. In fact, truncation is not necessary for
absolute convergence for $I_{\Omega_{1,0}} (\xi_{c,s})$.  Let $\xi =
\xi_s^c$. Via a similar argument we reach the step where we need to
show convergence of
$\sum_{\overline {\eta} \in \overline {N} (U_1,Fe_1^-)}
|a_{\overline {\eta} }|_\A^{-s+\rho_{Q_1}}\hat {\tau}^c (H ( m_1
(a_{\overline {\eta} })))$.
This is a finite sum because of truncation. Thus
convergence is trivial.

\begin{lemma}
  For $\Re s$ large enough, $I_{\Omega_{1,0}}(\xi_{\infty,s})$ (i.e., without truncation) is absolutely convergent. For any $s$ where $M (w,s)$ is defined, $I_{\Omega_{1,0}}(\xi_s^c)$ is absolutely convergent.
\end{lemma}

\subsection{Absolute Convergence of $I_{\Omega_{1,1}}$}
\label{sec:I-Omega11-abs-conv} We study the absolute convergence of $I_{\Omega_{1,1}}$ (c.f. \eqref{eq:I-Omega11}). We make use of notation set up in
Sec.~\ref{sec:I-Omega10-abs-conv}. Thus we also fix $v_0, v_0^+$ in
$V$ and make choices of $\gamma_{Fv}$ as in
Sec.~\ref{sec:I-Omega10-abs-conv}. Let $U$ be the subspace of $V$ such
that $V = Fv_0^+ \oplus U \oplus Fv_0$. For $t\in \GL_1$, we let $m_1'
(t)$ denote the element in $G (Fv_0^+\oplus Fv_0)$ given by
$\smatrix{t}{0}{0}{t^{-1}}$.
Write $\gamma_0$ for $\gamma_{F (v_0+e_1^-)}$.  We need
to consider absolute convergence for
$\int_{J (Z_1,e_1^-) \lmod G (Z_1) (\A)} \sum_{\eta \in
Q (V,Fv_0)} \xi (\gamma_0\eta g)dg$, 
which is formally equal to
\begin{equation*} \int_{K_{G (Z_1)}} \int_{[N (Z_1,Fe_1^-)] }
\int_{\GL_1 (\A)} \int_{[G (Z)]}\sum_{\eta \in Q (V,Fv_0) \lmod G (V)}
\xi (\gamma_0\eta ng m_1 (t) k ) |t|_\A^{-2\rho_{Q (Z_1,Fe_1^-)}} dg dt
dn dk.
\end{equation*} First let $\xi = \xi_{c,s}$. As observed in
Sec.~\ref{sec:I-Omega10-abs-conv}, the theta part is invariant under
$G (X_1) (F)$.  We have bound for the theta part by
Lemma~\ref{lemma:theta-bound}:
\begin{align*} |\theta_{X_1,Y} (n g m_1 (t) k) | < C (|t|_\A^{\dim
Y/2} + |t|_\A^{-R_1}) ||g||^{R_2}
\end{align*} for some constants $C, R_1, R_2>0$. Thus it suffices to
consider
\begin{align*} \int_{K_{G (Z_1)}} \int_{[N (Z_1,Fe_1^-)] } \int_{\GL_1
(\A)} \int_{[G (Z)]} \sum_{\eta \in Q (V,Fv_0) \lmod G (V)} f_s
(\gamma_0\eta ng m_1 (t) k ) \\ \hat {\tau}_c (\gamma_0\eta ng m_1 (t)
k ) |t|^{R} ||g||^{R_2} dg dt dn dk.
\end{align*} for some constants $R \in \R$ and $R_2 >0$.  Note that
$\gamma_0\eta n gm_1 (t) = n'g \gamma_0 m_1 (t) \eta$
for $n' \in N (Fv_0^+ \oplus Z \oplus Fv_0,Fv_0) (\A)$.  Thus
we are reduced to showing absolute convergence for
\begin{align*} & \int_{K_{G (Z_1)}} \int_{[N (Fv_0^+ \oplus Z \oplus
Fv_0,Fv_0)] } \int_{\GL_1 (\A)} \int_{[G (Z)]} \sum_{\eta \in Q (V,Fv_0)
\lmod G (V)}f_s (n'g \gamma_0 m_1 (t) \eta k )\\ &\quad \hat {\tau}_c
(H (\gamma_0 m_1 (t)\eta)) |t|_\A^{R} ||g||^{R_2} dg dt dn' dk.
\end{align*} As in Sec.~\ref{sec:I-Omega10-abs-conv} we use
$\{1\}\sqcup \overline {N} (V,Fv_0)w_0 $ as representatives for $ Q
(V,Fv_0) \lmod G (V)$ and as $w_0\in K_{G (X_1)}$ we just need to
consider absolute convergence for
\begin{align*} & \int_{K_{G (Z_1)}} \int_{[N (Fv_0^+ \oplus Z \oplus
Fv_0,Fv_0)] } \int_{\GL_1 (\A)} \int_{[G (Z)]} \sum_{\overline {\eta} \in
\overline {N} (V,Fv_0) }f_s (n'g \gamma_0 m_1 (t) \overline {\eta} k
)\\ &\quad \hat {\tau}_c (H (\gamma_0 m_1 (t)\overline {\eta}))
|t|_\A^{R} ||g||^{R_2} dg dt dn' dk.
\end{align*}
Write
$\overline {\eta} = n_{\overline {\eta}}' m_1'
(a_{\overline {\eta}})h_{\overline {\eta}}k_{\overline {\eta}}
$
in the Iwasawa decomposition
\begin{align*} G (V) (\A) = N (V,Fv_0) (\A) M (V,Fv_0) (\A) K_{G (V)}.
\end{align*} Here $M (V,Fv_0) \isom \GL_1 \times G (U)$. Then
$$\gamma_0 m_1 (t)\overline {\eta} = \gamma_0
n_{\overline {\eta}}' m_1' (a_{\overline {\eta}})h_{\overline {\eta}}
m_1 (t) k_{\overline {\eta}}= (\gamma_0 n_{\overline {\eta}}'
\gamma_0^{-1}) \gamma_0 m_1' (a_{\overline {\eta}})h_{\overline
{\eta}} m_1 (t) k_{\overline {\eta}}.
$$
If we write
$n_{\overline {\eta}}' =
  \begin{pmatrix} 1 & \mu &\alpha\\ & I & \nu \\ && 1
  \end{pmatrix}$,
with respect to the `basis' $v_0^+, U, v_0$, then
$\gamma_0 n_{\overline {\eta}}' \gamma_0^{-1} =
u_{\overline {\eta}} (n_{\overline {\eta}}')^{-1}
$
with
\begin{align*} u_{\overline {\eta}} =
  \begin{pmatrix} 1 & 0 & \mu &-\alpha &\alpha\\ & 1 &0 &0 &-\alpha\\
& &I &0 &\nu\\ & & &1 &0\\ &&&&1
  \end{pmatrix}
\end{align*} with respect to the `basis' $e_1^+,v_0^+, U,
v_0,e_1^-$. After conjugating $n'g$ over $u_{\overline {\eta}}$, it
drops out of $f_s$ and $\hat {\tau}_c$. We also note that
\begin{align*} \gamma_0 m_1' (a_{\overline {\eta}})h_{\overline
{\eta}} m_1 (t) = h_{\overline {\eta}} m_1 (a_{\overline {\eta}})m_1'
(t) \overline {u} (-a_{\overline {\eta}}t^{-1})\gamma_1k_{\overline
{\eta}}
\end{align*} where $\overline {u} (-a_{\overline {\eta}}t^{-1})$ is
the element that acts on $e_1^+, v_0^+$ by
  $\smatrix{1}{0}{-a_{\overline {\eta}}t^{-1}}{1}$
and on $v_0, e_1^-$ by duality. We change variable
$t\mapsto -a_{\overline {\eta}}t^{-1}$ and thus the integral under
consideration becomes
\begin{align*} & \int_{K_{G (Z_1)}} \int_{[N (Fv_0^+ \oplus Z \oplus
Fv_0,Fv_0)] } \int_{\GL_1 (\A)} \int_{[G (Z)]} \sum_{\overline {\eta} \in
\overline {N} (V,Fv_0) } |t|_\A^{R} ||g||^{R_2} \\ &f_s (n'g (n_{\overline {\eta}}')^{-1}
h_{\overline {\eta}}m_1 (a_{\overline {\eta}})m_1' (t) \overline {u}
(-a_{\overline {\eta}}t^{-1})\gamma_1k_{\overline {\eta}}k)
\hat {\tau}_c (H (m_1 (a_{\overline {\eta}})m_1' (t) \overline {u}
(-a_{\overline {\eta}}t^{-1}))) dg dt dn' dk \\
=& \int_{K_{G (Z_1)}} \int_{[N (Fv_0^+ \oplus Z \oplus Fv_0,Fv_0)] }
\int_{\GL_1 (\A)} \int_{[G (Z)]} \sum_{\overline {\eta} \in \overline
{N} (V,Fv_0) } |a_{\overline {\eta}}t^{-1}|_\A^{R} ||g||^{R_2} \\&f_s (m_1 (a_{\overline {\eta}})n'g (n_{\overline
{\eta}}')^{-1} h_{\overline {\eta}}m_1' (-a_{\overline {\eta}}t^{-1})
\overline {u} (t)\gamma_1k_{\overline {\eta}} k) \hat
{\tau}_c (H (m_1 (a_{\overline {\eta}})\overline {u} (t)))
dg dt dn' dk.
\end{align*} The Iwasawa decomposition of $\overline {u} (t)$ can be
make explicit. Write $t=t_\infty t_1 t_2$ where $t_\infty$ is the
infinite part of $t$, $t_1$ consists of finite local components $t_v$
of $t$ such that $|t_v|_v \le 1$ and $t_2$ consists of finite local
components of $t$ such that $|t_v|_v > 1$. Then
\begin{align*}
  \begin{pmatrix} 1 & \\ t & 1
  \end{pmatrix} =&
  \begin{pmatrix} 1 & \overline {t}_\infty T_\infty^{-1} \\ & 1
  \end{pmatrix}
  \begin{pmatrix} T_\infty^{-1} & \\ & T_\infty
  \end{pmatrix}
  \begin{pmatrix} T_\infty^{-1} & -\overline {t}_\infty
T_\infty^{-1}\\ t_\infty T_\infty^{-1} & T_\infty^{-1}
  \end{pmatrix} \\ & \begin{pmatrix} 1 &\\ t_1 &1
  \end{pmatrix}
  \begin{pmatrix} 1 & t_2^{-1} \\ & 1
  \end{pmatrix}
  \begin{pmatrix} t_2^{-1} & \\ & t_2
  \end{pmatrix}
  \begin{pmatrix} & -1\\ 1 & t_2^{-1}
  \end{pmatrix}
\end{align*} where $T_\infty = (t_\infty\overline {t}_\infty+1)^\half$
and bar denotes complex conjugation. Thus we have
\begin{align*} & \int_{K_{G (Z_1)}} \int_{[N (Fv_0^+ \oplus Z \oplus
Fv_0,Fv_0)] } \int_{\GL_1 (\A)} \int_{[G (Z)]} \sum_{\overline {\eta} \in
\overline {N} (V,Fv_0) } |a_{\overline
{\eta}}T_\infty^{-1}t_2^{-1}|_\A^{s+\rho_{Q_1}} \\ &f_s (n'g
(n_{\overline {\eta}}')^{-1} h_{\overline {\eta}}m_1' (-a_{\overline
{\eta}}t^{-1} T_\infty t_2) k_t\gamma_1k_{\overline {\eta}} k)\\
&\quad \hat {\tau}_c (H (m_1 (a_{\overline
{\eta}}T_\infty^{-1}t_2^{-1}))) |a_{\overline {\eta}}t^{-1}|_\A^{R}
||g||^{R_2} dg dt dn' dk.
\end{align*} where $k_t\in K_{G (X_1)}$ is determined by $t$.  We have
bound from Lemma~\ref{lemma:rapid-dec-several-variable}: for any
$R_3>0$ there exists $C_3>0$ such that
\begin{align*} |f_s (n'g (n_{\overline {\eta}}')^{-1} h_{\overline
{\eta}}m_1' (-a_{\overline {\eta}}t^{-1} T_\infty t_2)
k_t\gamma_1k_{\overline {\eta}} k)| < C_3 m_{P_0} (g m_1' (-a_{\overline
{\eta}}t^{-1} T_\infty t_2))^{-R_3 \rho_{P_0}} \\ = C_3 m_{P_0}(g
)^{-R_3 \rho_{P_0}}|a_{\overline {\eta}}t^{-1} T_\infty
  t_2|_\A^{-R_3\dim X}
\end{align*} where $m_{P_0}$ is the map $G (X) (\A) \rightarrow
M_0 (\A)^1 \lmod M_0 (\A)$ as defined in \cite{MR1361168} for
the minimal parabolic subgroup $P_0$ of $X$. Thus when $R_3$ is large
enough integration over $[G (Z)]$ is absolutely convergent. The summation
over $\overline {\eta}$ is absolutely convergent for $\Re s$ large
enough because it is bounded by an Eisenstein series as is the case in
Sec.~\ref{sec:I-Omega10-abs-conv}. For `$t$'-part we need to check
convergence of
\begin{equation*} \int_{\A^\times}|T_\infty^{-1} t_2^{-1}
|_\A^{s+\rho_{Q_1}+R_3 \dim X} |t|_\A^{R_3 \dim X -R } d^\times t
\end{equation*} We set $C_4 = \rho_{Q_1}+R_3 \dim X$ and $C_5 = R_3
\dim X -R$. We pick $R_3$ such that $C_5 > 1$.
To define height of a vector, temporarily let $X$ be any vector space over $F$ with a fixed
basis. Let $x\in X (\A)$. We define $||x||_v$ to be the maximal norm
of the coordinates of $x_v$ if $v$ is finite, the usual Euclidean norm
if $v$ is real and the square of the usual Euclidean norm if $v$ is
complex. Then define $||x||=\prod_v ||x||_v$. Then 
$|T_\infty^{-1} t_2^{-1} |_\A$ is comparable to
$||e_1^- \overline {u} (t)||^{-1} = || e_1^- - t v_0
||^{-1}$.
Thus we just need to
show that 
$\int_{\A^\times}|| e_1^- - t v_0 ||^{- (s+C_4)}
|t|_\A^{C_5} d^\times t$ converges. 
Consider the local integrals.  For a finite place $v$ of
$F$
\begin{align*} &\int_{F_v^\times} || e_1^- - t v_0 ||_v^{- (s+
C_4)}|t|_\A^{C_5} d^\times t  = \int_{F_v^\times} \max\{1,
|t|_v\}^{- (s+ C_4)} |t|_\A^{C_5} d^\times t \\
=&\int_{\substack{F_v^\times\\ |t|_v \le 1}} |t|_\A^{C_4} d^\times t +
\int_{\substack{F_v^\times\\|t|_v > 1}} |t|_v^{- (s+ C_4 -C_5)}
d^\times t =\frac{L_v (C_5) L_v (s+ C_4 -C_5)}{L_v (s+ C_4)}.
\end{align*} Thus the product over all finite places of $F$ is
absolutely convergent for $\Re s$ large enough. For a real place $v$
of $F$,
\begin{equation*}
  \int_{F_v^\times} || e_1^- - t v_0 ||_v^{- (s+
C_4)}|t|_\A^{C_5} d^\times t = \int_{F_v^\times}
(1+|t|_v^2)^{-\half (s+ C_4)}|t|_\A^{C_5} d^\times t
\end{equation*}
while for a complex place $v$ of $F$,
\begin{equation*}
\int_{F_v^\times} || e_1^- - t v_0 ||_v^{- (s+
C_4)}|t|_\A^{C_5} d^\times t  =\int_{F_v^\times} (1+|t|_v)^{-(s+
C_4)}|t|_\A^{C_5} d^\times t.
\end{equation*}
They are both absolutely convergent for $\Re s$ large
enough. This shows that $I_{\Omega_{1,1}} (\xi_{c,s})$ is absolutely
convergent for $\Re s$ large enough. In fact truncation is not needed
for absolute convergence.

Next let $\xi = \xi_s^c$. Via an analogous computation we reach the
point where we need to show absolute convergence for
\begin{align*} & \int_{K_{G (Z_1)}} \int_{[N (Fv_0^+ \oplus Z \oplus
Fv_0,Fv_0)] } \int_{\GL_1 (\A)} \int_{[G (Z)]} \sum_{\overline {\eta} \in
\overline {N} (V,Fv_0) } |a_{\overline
{\eta}}T_\infty^{-1}t_2^{-1}|_\A^{-s+\rho_{Q_1}} \\ &M (w,s) f_s (n'g
(n_{\overline {\eta}}')^{-1} h_{\overline {\eta}}m_1' (-a_{\overline
{\eta}}t^{-1} T_\infty t_2) k_t\gamma_1k_{\overline {\eta}} k)\\
&\quad \hat {\tau}^c (H (m_1 (a_{\overline
{\eta}}T_\infty^{-1}t_2^{-1}))) |a_{\overline {\eta}}t^{-1}|_\A^{R}
||g||^{R_2} dg dt dn' dk.
\end{align*} In this case truncation is essential. We note that
$|a_{\overline {\eta}}T_\infty^{-1}t_2^{-1}|_\A$ is comparable to
\begin{align*} ||e_1^-\gamma_1\overline {\eta}\gamma_1^{-1} \overline
{u} (t)||^{-1} = ||e_1^-\gamma_1\overline {\eta}\gamma_1^{-1} -
tv_0||^{-1}.
\end{align*} Thus there exists some constant $c'>0$ such that only
those terms with $||e_1^-\gamma_1\overline {\eta}\gamma_1^{-1} -
tv_0||^{-1} > c'c$ are preserved. On the other hand
\begin{align*} ||e_1^-\gamma_1\overline {\eta}\gamma_1^{-1} - tv_0||
\ge ||e_1^-\gamma_1\overline {\eta}\gamma_1^{-1}|| \ge ||e_1^-||.
\end{align*} Thus if we choose $c$ large enough none of the terms is
preserved. In this case we obviously have absolute convergence.

\begin{lemma}
  For $\Re s$ large enough, $I_{\Omega_{1,1}}(\xi_{\infty,s})$ (i.e., without truncation) is absolutely convergent. For large enough truncation parameter $c$, $I_{\Omega_{1,1}}(\xi_s^c)$ is absolutely convergent for any $s$ where $M (w,s)$ is defined and it has value $0$.
\end{lemma}

\bibliographystyle{alpha} \bibliography{MyBib}
\end{document}